\documentclass[sn-mathphys-num]{sn-jnl}
\usepackage{amsmath,amssymb}
\usepackage{tikz}
\usepackage{algpseudocode}

\usepackage{color, booktabs}
\usepackage{mathtools}
\usepackage{xspace}
\usepackage{microtype}
\usepackage{subcaption}

\usepackage{LatexDefinitions}
\floatstyle{ruled}
\newfloat{algorithmfloat}{t}{lop}
\floatname{algorithmfloat}{Algorithm}

\newcommand{\proj}{\tn{P}}
\newcommand{\cont}{\mc{C}}

\newcommand{\partA}{\mc{J}_A}
\newcommand{\partB}{\mc{J}_B}

\newcommand{\partJ}{\mc{J}}

\newcommand{\partk}{\iter{\partJ}{k}}

\newcommand{\piInit}{\pi_{\tn{init}}}

\newcommand{\floor}[1]{\lfloor #1 \rfloor}

\newcommand{\Lebesgue}{\mathcal{L}}

\newcommand{\dom}{\operatorname{dom}} 
\newcommand{\LL}{\mathrm{L}}

\newcommand{\iterA}{\mathcal{F}_A}
\newcommand{\iterB}{\mathcal{F}_B}
\newcommand{\partGeneric}{\mathcal{J}}

\newcommand{\GetWeights}{\textsc{GetWeights}\xspace}
\newcommand{\GetWeightsSub}[1]{\textsc{GetWeights}$_{ \textsc{#1}}$}

\usepackage{ulem}

\newcommand{\DomDecGPU}{\textsc{DomDec}}
\newcommand{\SinkhornGPU}{\textsc{Sinkhorn}}

\newcommand{\PD}{\mathrm{PD}}

\author{\fnm{Ismael} \sur{Medina}}

\author{\fnm{The Sang} \sur{Nguyen}}

\author*{\fnm{Bernhard} \sur{Schmitzer}$^*$}\email{schmitzer@cs.uni-goettingen.de}

\affil{\orgdiv{Institute of Computer Science}, \orgname{University of Göttingen}, \orgaddress{\street{Goldschmidstraße} 7, \city{Göttingen}, \postcode{37077}, \country{Germany}}}

\date{}                    
\title{Domain decomposition for entropic unbalanced optimal transport}

\begin{document}
	
\maketitle
	
\begin{abstract}
	
Entropic optimal transport has become a popular tool in data analysis and it can be solved efficiently with the celebrated Sinkhorn algorithm, but large scale problems remain challenging. Domain decomposition has been shown to be an efficient strategy on large grids.
Unbalanced optimal transport is a versatile generalization of the standard (balanced) optimal transport problem and its entropic variant can also be solved with a generalized Sinkhorn algorithm.
However, the domain decomposition algorithm cannot be applied directly to the unbalanced problem since independence of the cell problems is lost.
In this article we generalize the domain decomposition algorithm for optimal transport to the unbalanced setting by introducing a new adaptive step size strategy, which allows to ensure the decrement of the global score and prove convergence to the global minimizer.
We also provide an efficient GPU implementation of the new algorithm and demonstrate with experiments that domain decomposition is also an efficient strategy for large unbalanced optimal transport problems. 
\end{abstract}

\section*{Acknowledgements and competing interests}
IM and BS were supported by the Emmy Noether programme of the Deutsche Forschungsgemeinschaft, project number 403056140.
The authors declare no competing interests.

\section{Introduction}

\subsection{Motivation}
\label{sec:motivation}
\paragraph{(Computational) optimal transport and applications} 
Optimal transport (OT) is concerned with finding the most cost-efficient way of transforming one probability measure into another. Let $X$ and $Y$ be compact metric spaces. For a measurable cost function $c: X\times Y \to \R$ and two probability measures $\mu$ and $\nu$ on $X$ and $Y$ the Kantorovich OT problem is given by
\begin{equation}
\label{eq:unregularized-ot}
\inf_{\pi \in \Pi(\mu, \nu)}
\la c, \pi \ra,
\end{equation}
where $\la\cdot, \cdot \ra$ denotes integration of measurable functions against measures and
\begin{equation}
\label{eq:coupling}
	\Pi(\mu, \nu) 
	\assign
	\{
	\gamma \in \measp(X\times Y) 
	\mid 
	\proj_X \pi = \mu, \ 
	\proj_Y \pi = \nu
	\}.
\end{equation}
denotes the set of \emph{couplings} or \emph{transport plans} between $\mu$ and $\nu$. Here $\proj_X$ and $\proj_Y$ are the marginal projection operators that extract the $X$ and $Y$ marginals of $\pi$.
Intuitively, an element of $\Pi(\mu, \nu)$ can be interpreted as a description of how to rearrange the mass of $\mu$ into $\nu$, where (formally) $\pi(x,y)$ describes the infinitesimal amount of mass moving from $x$ to $y$. Note that definition \eqref{eq:coupling} is valid for any pair of non-negative measures $\mu$ and $\nu$ with equal total mass.

Thorough introductions to the topic can be found in \cite{Villani-OptimalTransport-09} and \cite{SantambrogioOT}, and an overview on computational methods can be found in \cite{PeyreCuturiCompOT}.

A common variant of \eqref{eq:unregularized-ot} is the addition of \emph{entropic regularization}:
\begin{equation}
\label{eq:regularized-ot}
\inf_{\pi\in\Pi(\mu, \nu)}
\la c, \pi \ra
+
\veps \KL(\pi \mid \mu \otimes \nu)
\end{equation}
where $\KL$ is the Kullback-Leibler diveregence (see Section \ref{sec:notation} for a full definition) and $\veps > 0$ is a positive \emph{regularization parameter}.
Here $\mu \otimes \nu$ denotes the product measure of $\mu$ and $\nu$ on $X \times Y$.
This modification has several theoretical and computational advantages with respect to \eqref{eq:unregularized-ot}. For instance, the optimization problem becomes strictly convex, and thus amenable to efficient convex optimization solvers. 
In particular, the \emph{Sinkhorn algorithm} \cite{Cuturi2013} has become the de facto standard for solving \eqref{eq:regularized-ot}, among other reasons thanks to its fast convergence \cite{Berman-Sinkhorn-2017} and amenability to GPU parallelization \cite{Cuturi2013,Solomon-siggraph-2015, PeyreCuturiCompOT}.
Entropic regularization also reduces the sample complexity of standard OT (which suffers the curse of dimensionality) down to the parametric rate \cite{sample-complexity-ot, sample-complexity-entropic-ot}, and smooths the gradients of \eqref{eq:regularized-ot} with respect to the problem data, advancing the integration of OT into learning pipelines \cite{feydy_geomloss}.

These advances have facilitated the adoption of optimal transport as a valuable data science tool. Examples are its use in inference on stochastic processes, e.g.~allowing to recover the collective population dynamics from finite samples \cite{JKONet2021} or even reconstructing the underlying drift term \cite{Lavenant2021}.
In biology and medicine, optimal transport has found applications in the study of single-cell genomics \cite{WaddingtonOT,HuiPeyCanOmics22,single-cell-genomics}, histology \cite{OptimalTransportTangent2012}, protein colocalization \cite{TamelingColoc2021}, and dynamic PET imaging \cite{ScScWi19}.

\paragraph{Unbalanced optimal transport} 
One limitation of problems \eqref{eq:unregularized-ot} and \eqref{eq:regularized-ot} is that they can only be applied to positive measures $\mu$ and $\nu$ of equal mass, and that they can be quite susceptible to additive noise. As a remedy, a general framework for \emph{unbalanced} optimal transport between measures of (possibly) different mass has been proposed, see for instance \cite{Frogner2015,CHIZAT20183090,Liero2018} and references therein. In \cite{Liero2018} the constraint that the marginals of $\pi$ must be exactly equal to $\mu$ and $\nu$ is relaxed through \emph{entropy functions} (also known as \emph{divergence functionals}) that penalize the deviations from $\mu$ and $\nu$. The general form of the \emph{unbalanced optimal transport} (UOT) problem becomes
\begin{equation}
\label{eq:unbalanced-ot}
\inf_{\pi \in \measp(X\times Y)}
\la c, \pi \ra
+
D_1(\proj_X \pi \mid \mu)
+
D_2(\proj_Y \pi \mid \nu)
\end{equation}
where $D_1$ and $D_2$ are convex functionals satisfying some properties to be specified in Section \ref{sec:background}. Typically, the terms $D_1(\proj_X \pi \mid \mu)$ and $D_2(\proj_Y \pi \mid \nu)$ encourage the marginals of $\pi$ to be close to $\mu$ and $\nu$, respectively, without enforcing a strict equality constraint.
It is natural to also consider the entropy regularized variant of \eqref{eq:unbalanced-ot}:
\begin{equation}
\label{eq:unbalanced-entropic-ot}
\inf_{\pi \in \measp(X\times Y)}
\la c, \pi \ra
+
\veps \KL(\pi | \mu \otimes \nu)
+
D_1(\proj_X \pi \mid \mu)
+
D_2(\proj_Y \pi \mid \nu).
\end{equation}
It is shown in \cite{ChizatUnbalanced2016} that such problems can be solved with a generalized version of the Sinkhorn algorithm, which preserves its favorable computational properties (see Definition \ref{def:sinkhorn} below).
The favorable properties of unbalanced OT have been applied to fields such as shape registration in biology \cite{uot-bio}, LiDAR imaging \cite{uot-lidar} and the theoretical analysis of neural networks \cite{uot-nn}.

\paragraph{Domain decomposition for optimal transport} 
Despite its general efficiency, the Sinkhorn algorithm struggles to solve large problems with high precision. One way to address this is domain decomposition \cite{BoSch2020}, originally proposed for unregularized optimal transport in \cite{BenamouPolarDomainDecomposition1994}. The strategy of the algorithm is as follows: Let $\partA$ and $\partB$ be two partitions of $X$, and let $\pi^0 \in \Pi(\mu, \nu)$ be an initial feasible coupling. Alternating between $\partA$ and $\partB$, one optimizes the restriction of $\pi$ to each set of the form $X_J\times Y$, where $X_J$ are the subdomains in the current partition, leaving the marginals on each of the restrictions fixed (a more detailed description of the algorithm is given in Section \ref{sec:domdec}). This can be done independently for each subdomain in the partition, so the algorithm can be easily parallelized.

\cite{BenamouPolarDomainDecomposition1994} and \cite{asymptotic_domdec_arxiv} show convergence of domain decomposition for unregularized optimal transport \eqref{eq:unregularized-ot} with the squared Euclidean distance, under different assumptions on the overlap between the partition cells. \cite{BoSch2020} observes that much weaker assumptions are sufficient when adding entropic regularization and shows linear convergence for more general costs, marginals, and partitions.
\cite{BoSch2020} and \cite{hybrid} also describe details of efficient implementations with a multiscale strategy and CPU or GPU parallelization, respectively, and demonstrate that domain decomposition clearly outperforms state-of-the-art implementations of the Sinkhorn algorithm on large problems, in terms of runtime, memory footprint, and solution precision.

\subsection{Outline}
This article generalizes domain decomposition to \emph{unbalanced} entropic optimal transport. The main difficulty compared to the balanced version is that the sub-problems on the partition cells $X_J \times Y$ for \eqref{eq:unbalanced-entropic-ot} are no longer independent of each other, since they interact through the flexible penalty on the $Y$-marginal, and can therefore not naively be solved in parallel.
We show that this issue can be overcome by suitable post-processing of the parallel iterations.

Our analysis is structured as follows: The mathematical assumptions and basic prerequisites on unbalanced entropic optimal transport and domain decomposition are collected in Section \ref{sec:background}.

\paragraph{Domain decomposition for unbalanced optimal transport} Section \ref{sec:unbalanced-domdec} studies the adaptation of domain decomposition to unbalanced entropic optimal transport. Section \ref{sec:formulation-sequential-parallel} describes the main difficulties imposed by the soft marginal constraints and proposes both a sequential and a parallel algorithm, where the latter generates new iterates as convex combinations between the previous iterate and the minimizers of the cell subproblems, controlled by an adaptive step size parameter. Section \ref{sec:properties-cell-subproblem} establishes relevant regularity properties of the cell subproblems. Section \ref{sec:convergence-sequential} shows convergence of the sequential algorithm, i.e. without parallelization.

\paragraph{Parallelization of the domain decomposition iterations} 

In Section \ref{sec:domdec-parallel} we give a sufficient condition on the adaptive step size strategy for convergence of
the \emph{parallel} algorithm.
We show that greedy update rules provide global convergence, as long as their score decrement is not worse than that of a suitable convex combination of previous and new iterates, thus combining efficiency and theoretical robustness.

\paragraph{Implementation details and numerical experiments} 
Section \ref{sec:implementation} describes some details of an efficient implementation, focusing on the $\KL$-divergence as soft-marginal constraint. 
The different parallelization strategies proposed in Section \ref{sec:domdec-parallel} are compared experimentally in Section \ref{sec:single-scale}.
Numerical results on large-scale experiments are presented in Section \ref{sec:multiscale}. While unbalanced domain decomposition is more involved than the balanced variant, we still observe a considerable gain of performance relative to a single global unbalanced Sinkhorn solver on large problems.

\section{Background}
\label{sec:background}

\subsection{Setting and notation}
\label{sec:notation}
\begin{itemize}
	\item Let $X$ and $Y$ be compact metric spaces. 
	We assume compactness to avoid overly technical arguments while 
	covering the numerically relevant setting. 
	\item For a compact metric space $Z$ denote by $\cont(Z)$ the set of continuous functions on $Z$ and by $\meas(Z)$ the set of finite 
	signed Radon measures on $Z$. 
	We identify $\meas(Z)$ with the topological dual of $\cont(Z)$, and we recall that integration against test functions in $\cont(Z)$ induces the weak* topology on $\meas(Z)$, i.e., $(\mu_n)_n$ converges weak* to $\mu$ if, and only if, $\int_Z \phi \diff \mu_n \to \int_Z \phi \diff \mu$ for all continuous $\phi$.
	The subsets of non-negative and probability measures are denoted by 
	$\meas_+(Z)$ and $\meas_1(Z)$, respectively.
	The total variation norm of a measure $\mu\in\meas(Z)$ is denoted by 
	$\|\mu\|_{\meas(Z)}$. 
	One has $\|\mu\|_{\meas(Z)} = \mu(Z)$ for $\mu\in\meas_+(Z)$. 
	We will simply write $\|\mu\|$ for the total variation norm if the underlying 
	space is clear from the context. 
	\item For $\mu, \nu \in \meas(Z)$ we write $\mu \ll \nu$ to indicate that $\mu$ is absolutely continuous with respect to $\nu$. 
	\item For $\mu\in\meas_+(Z)$ we denote by $L^1(Z,\mu)$ and $L^\infty(Z,\mu)$ 
	the spaces of $\mu$-integrable and essentially bounded functions on $Z$,
	respectively. 
	The corresponding norms are denoted by 
	$\|\cdot\|_{L^1(Z,\mu)}$ and $\|\cdot\|_{L^\infty(Z,\mu)}$, but we 
	often merely write $\|\cdot\|_{1, \mu}$ and 
	$\|\cdot\|_{\infty, \mu}$ for simplicity.
	\item We denote by $\la \phi, \rho \ra$ the integration of a measurable function $\phi$ against a measure $\rho$. 
	\item For $\mu\in\meas_+(Z)$ and a measurable set $S\subseteq Z$ the restriction 
	of $\mu$ to $S$ is denoted by $\mu\restr S$. 
	\item The maps $\proj_X\colon \meas_+(X\times Y)\to\meas_+(X)$ and 
	$\proj_Y\colon \meas_+(X\times Y)\to\meas_+(Y)$ denote the projections of 
	measures on $X\times Y$ to their marginals, i.e. 
	\[
	(\proj_X\pi)(A) := \pi(A\times Y) 
	\hspace{0.5cm} \text{and} \hspace{0.5cm}
	(\proj_Y\pi)(B) := \pi(X\times B)
	\]
	for all measurable $A\subseteq X$ and $B\subseteq Y$, 
	$\pi\in\meas_+(X\times Y)$. 
	\item For (measurable) functions 
	$\alpha\colon X\to\R\cup\{\infty\}, \beta\colon Y\to\R\cup\{\infty\}$, 
	the function$\alpha\oplus\beta\colon X\times Y\to\R\cup\{\infty\}$
	is defined
	by
	\[
	(\alpha\oplus\beta)(x,y) := \alpha(x) + \beta(y).
	\hspace{0.5cm} \text{} \hspace{0.5cm}
		\]
		For two measures $\mu\in\meas_+(X)$ and $\nu\in\meas_+(Y)$, their product 
	measure is denoted by $\mu\otimes\nu\in\meas_+(X\times Y)$.
	
	\item We denote by $f^*$ the convex conjugate of a real-valued function $f$.
	\item For two non-negative measures $\mu, \nu \in \meas_+(Z)$, we define the Kullback-Leibler divergence as 
\begin{align}
\label{eq:KL}
\KL(\mu|\nu) \assign 
\int_Z
\left[\RadNikD{\mu}{\nu} \log \left(\RadNikD{\mu}{\nu}\right)
-
\RadNikD{\mu}{\nu} 
+ 
1\right] \diff \nu
\end{align}
whenever $\mu \ll \nu$, and $+\infty$ otherwise.
\end{itemize}

\subsection{Entropic unbalanced optimal transport}
\label{sec:background-ot}
In this section we provide some additional technical details and background on entropic unbalanced optimal transport, building on the introduction in Section \ref{sec:motivation}.

The precise notion of entropy functions that we consider for \eqref{eq:unbalanced-ot} is given by the following definition, adapted from \cite{Liero2018}.
\begin{definition}[Entropy function and divergence]
	\label{def:divergence}
	A function $\varphi\colon \R \to \R\cup\{\infty\}$ is called
	\emph{entropy function} if it is lower semicontinuous, convex, 
	$\dom(\varphi) \subseteq [0,\infty)$, and the set $(\dom\varphi) \backslash \{0\}$ is non-empty.
	
	The speed of growth of $\varphi$ at $\infty$ is described by 
	$
	\varphi'_\infty 
	:= \lim_{x\nearrow +\infty} \frac{\varphi(x)}{x}.
	$
	If $\varphi'_\infty=\infty$, $\varphi$ is said to be \emph{superlinear}
	as it grows faster than any linear function. 
	
	Let $\varphi$ be an entropy function. For $\rho, \sigma\in\meas(Z)$, let 
	$\rho=\frac{\diff\rho}{\diff\sigma}\sigma + \rho^{\perp}$ be the Lebesgue 
	decomposition of $\rho$ with respect to $\sigma$. 
	Then the \emph{($\varphi$-) divergence} of $\rho$ with respect to $\sigma$ is defined
	as:
	\[
	D_\varphi(\rho \mid \sigma) 
	:= \int_Z \varphi\left(\frac{\diff\rho}{\diff\sigma}\right) \diff\sigma
	+ \varphi'_\infty\rho^{\perp}(Z).
	\]
\end{definition}

Divergence functionals enjoy convenient regularity properties:

\begin{proposition}[\cite{ChizatUnbalanced2016}, Proposition 2.3]\label{lsc}
	Let $\varphi$ be an entropy function. Then the divergence functional 
	$D_\varphi$ is convex and weakly* lower-semicontinuous in $(\rho,\sigma)$.
\end{proposition}

\begin{example}
	The convex indicator function $\iota_{\{=\}}(\rho\mid\sigma)$ (returning 0 if $\rho$ equals $\sigma$ and $+\infty$ otherwise) is a divergence functional, associated to the convex indicator function of the real set $\{1\}$.
	Likewise, the KL divergence \ref{eq:KL} is the divergence functional associated to the entropy function $\varphi_{\KL} \assign s \log s - s + 1$. Its convex conjugate is given by $\varphi_{\KL}^*(z) = \exp(z)-1$.
\end{example}
Note that by selecting $D_1 = D_2 = \iota_{\{=\}}$ in \eqref{eq:unbalanced-ot} we recover the original Kantorovich problem \eqref{eq:unregularized-ot}.

In preparation of our subsequent analysis, in the following formulation of the unbalanced entropic problem below we add an auxiliary \emph{background} measure $\nu_{-}$ in the definition of the second marginal penalty.
This background measure should be interpreted as a contribution to the target measure by a different part of the transport plan that we do not aim to (or cannot) optimize due to the domain decomposition, but that needs to be accounted for in the overall marginal penalty.

\begin{definition}[Unbalanced entropic optimal transport with background measure $\nu_{-}$]
	\label{def:unbalanced-entropic-transport}
	Given some entropy functions $\varphi_1: X \rightarrow \R_+$, $\varphi_2: Y \rightarrow \R_+$ (and their associated divergences $D_1$ and $D_2$), measures $\mu\in \measp(X)$, $\nu, \nu_{-} \in \measp(Y)$, a cost function $c$, and a regularization strength $\veps$, the \emph{primal} unbalanced entropic transport problem is given by:
	\begin{equation}
	\label{eq:unbalanced-problem}
	\begin{aligned}
	\inf_{\pi \in \measp(X \times Y)} 
	E(\pi\mid \nu_{-}),
	\quad
	\tn{ with }&
	E(\pi\mid\nu_{-})
	=
	\la c, \pi \ra 
	+
	\varepsilon \mathrm{KL}(\pi \mid \mu\otimes\nu)
	\\
	&\ \ +
	D_1\left(\mathrm{P}_X \pi \mid \mu\right)
	+
	D_2\left(\mathrm{P}_Y \pi + \nu_- \mid \nu\right) .
	\end{aligned}
	\end{equation}
	The corresponding \emph{dual} problem reads:
	\begin{equation}	
	\label{eq:dual-problem}
	\begin{aligned}
	\sup _{\substack{\alpha \in \cont(X) \\ \beta \in \cont(Y)}}
	D(\alpha, \beta \mid \nu_{-}),
	\quad 
	\tn{with }
	D(\alpha, &\beta \mid \nu_{-})
	=
	\varepsilon \int_{X \times Y} 1 - \exp \left(\tfrac{\alpha \oplus \beta - c}{\varepsilon}\right) \mathrm{d} (\mu\otimes\nu)
	\\
	&\ \ 
	-\int_X \varphi_1^*(-\alpha) \mathrm{d} \mu
	-\int_Y \varphi_2^*(-\beta) \mathrm{d} \nu
	- \la \beta , \nu_{-} \ra.
	\end{aligned}
	\end{equation}
\end{definition}

The following theorem establishes first the existence of minimizers of the primal problem \eqref{eq:unbalanced-problem}. Second, assuming the existence of maximizers of the dual problem \eqref{eq:dual-problem}, it shows how primal and dual solutions can be inferred from each other via primal-dual optimality conditions. This is particularly useful for efficient algorithmic treatment.
\begin{theorem}[\protect{\cite[Theorem 3.2 and Proposition 3.6]{ChizatUnbalanced2016}} Solutions to entropic UOT]
	\label{theorem:existence-global-solution}
	Assume that problem \eqref{eq:unbalanced-problem} is feasible (i.e.~the infimal value is finite). Then,
	\begin{enumerate}
		\item \label{item:existence-global-primal} there exists a unique minimizer $\pi^\dagger$ for the primal problem \eqref{eq:unbalanced-problem},
		
		\item \label{item:existence-global-dual} $(\alpha^\dagger, \beta^\dagger)$ are dual optimizers for  \eqref{eq:dual-problem} if and only if
	\end{enumerate}	
		\begin{equation}
		\nonumber
		\pi^\dagger
		=
		\exp \left(
		\frac{\alpha^\dagger \oplus \beta^\dagger - c}{\veps} 
		\right)\mu \otimes \nu
		\quad 
		\tn{ and }
		\quad
		\begin{cases}
		-\alpha^\dagger(x) \in \partial \varphi_1\left(
		\RadNik{\proj_X \pi^\dagger}{ \mu}(x)
		\right) &
		\tn{$\mu$-a.e.},
		\\
		-\beta^\dagger(y) \in \partial \varphi_2\left(
		\RadNik{(\proj_Y \pi^\dagger+\nu_{-})}{ \nu}(y)
		\right) &
		\tn{$\nu$-a.e.}
		\end{cases}
		\end{equation}
\end{theorem}
The feasibility condition is necessary for uniqueness of minimizers (otherwise any candidate is optimal). This detail was omitted in \cite[Theorem 3.2]{ChizatUnbalanced2016}.
Sufficient conditions for the existence of dual maximizers are given in \cite[Theorem 4.11]{multimarginal-uot}.

\begin{definition}[Sinkhorn algorithm for unbalanced entropic optimal transport]
	\label{def:sinkhorn}
	The Sinkhorn algorithm \cite{Cuturi2013,ChizatUnbalanced2016} is an alternate maximization scheme for the duals $\alpha$ and $\beta$ in \eqref{eq:dual-problem}. 
	For a suitable initialization for $\beta^0$ (e.g. $\beta^0 \equiv 0$), the Sinkhorn iterations are defined for each $k = 1,2,...$ as:
	
	\begin{equation}
	\label{eq:sinkhorn-iters-uot}
	\begin{gathered}
	\alpha^{k} 
	\in
	\argmax_{\alpha: X\to \R}
	D(\alpha, \beta^{k-1} \mid \nu_{-})
	,\quad
	\beta^{k} 
	\in
	\argmax_{\beta: Y \to \R}
	D(\alpha^{k}, \beta \mid \nu_{-})
	\\
	\pi^{k} = \exp\left(
	\frac{\alpha^{k} \oplus \beta^{k} - c}{\veps}
	\right) \mu \otimes \nu.
	\end{gathered}
	\end{equation}
	
	Convergence of the Sinkhorn iterates $(\pi^k)_k$ to the unique primal optimizer $\pi^\dagger$ is shown in \cite[Theorem 4.1]{unbalanced-chizat} for discrete marginals, and in \cite[Theorem 4.11]{multimarginal-uot} for arbitrary, compactly supported marginals (but choosing $\nu_{-} = 0$), both assuming the initialization is feasible.
	For several divergence functional of interest, the maximizers in each step of \eqref{eq:sinkhorn-iters-uot} can be given in closed form, see \cite[Table 1]{ChizatUnbalanced2016}. For example, choosing $D_1 = D_2 = \lambda \KL$, $\nu_{-}=0$, the iterations in \eqref{eq:sinkhorn-iters-uot} read:
	\begin{equation}
	\label{eq:sinkhorn-iters-ot}
	\begin{gathered}
	\alpha^{k}(x) 
	= - \frac{\veps}{1+\veps/\lambda}
	\left[
	\log
	\int_Y
	\exp
	\left(\frac{
		\beta^{k-1}(y) - c(x,y)
	}{\varepsilon}\right)
	\diff \nu(y)
	\right],
	\\
	\quad
	\beta^{k}(y) = 
	- \frac{\veps}{1+\veps/\lambda}
	\left[
	\log
	\int_X
	\exp\left(\frac{
		\alpha^{k}(x) - c(x,y)
	}{\varepsilon}\right)
	\diff \mu(x)
	\right].
	\end{gathered}
	\end{equation}
	and the limit as $\lambda \to \infty$ yields the iterates for the balanced case, i.e. $D_1 = D_2 = \iota_{\{=\}}$.
\end{definition}

When no closed forms are available, a nested maximization problem needs to be solved. Nevertheless, it can be shown that the maximization problem decomposes into a set of independent, one-dimensional convex optimization problems \cite{unbalanced-chizat}, which can be addressed efficiently. 
We will treat this case in the implementation section (Section \ref{sec:implementation}).
For now, we will work under the assumption that an efficient method exists for computing optimizers for \eqref{eq:unbalanced-problem}-\eqref{eq:dual-problem}.

\subsection{Domain decomposition for optimal transport}
\label{sec:domdec}

Domain decomposition for optimal transport was introduced in \cite{BenamouPolarDomainDecomposition1994} for unregularized transport with the squared Euclidean distance cost and in \cite{BoSch2020} for entropic transport for general costs.
It addresses the scalability issues of Sinkhorn-based algorithms \cite{SchmitzerScaling2019,Solomon-siggraph-2015,feydy_geomloss} ---which as shown in \cite{BoSch2020} struggle to approximate the unregularized solution on large problems --- with a divide-and-conquer approach.

Here we briefly recall the main definitions. From now on, $\mu, \nu$ will be two positive measures on $X$ and $Y$, respectively.

\begin{definition}[Basic and composite partitions]
	\label{def:partitions}
	 For some finite index set $I$, let $\{X_i\}_{i\in I}$ be a partition of $X$ into closed, $\mu$-essentially disjoint subsets, where $m_i \assign \mu(X_i)$ is positive for all $i\in I$. We call $\{X_i\}_{i\in I}$ the \emph{basic partition}, and $\{m_i\}_{i\in I}$ the \emph{basic cell masses}. The restriction of $\mu$ to basic cells, i.e.~$\mu_i \assign \mu\restr X_i$ for $i\in I$, are called the \emph{basic cell $X$-marginals}.
	Given a feasible plan $\pi$, we call $\nu_i 
\assign 
\proj_Y( \pi \restr (X_i\times Y))
$
the \emph{basic cell $Y$-marginals} of cell $i\in I$.
Note that due to the partition structure, the basic cell $Y$-marginals depend on the choice of $\pi$, whereas the basic cell $X$-marginals do not.

	 A \emph{composite partition} is a partition of $I$. For each $J$ in a composite partition $\partGeneric$, we define the \emph{composite cells}, and the \emph{composite cell $X$ and $Y$-marginals} respectively by 
		\begin{equation}
		X_J  \assign \bigcup_{i \in J} X_i, 
		\qquad 
		\mu_J  \assign \sum_{i \in J} \mu_i=\mu \restr X_J,
		\quad
		\nu_J  \assign \sum_{i \in J} \nu_i=\proj_Y( \pi \restr (X_J\times Y)).
		\end{equation}
		Note that the composite cell $Y$-marginals depend again on some given plan $\pi$. For simplicity we will consider two composite partitions $\partA$ and $\partB$. 
\end{definition}

The balanced domain decomposition algorithm for solving \eqref{eq:regularized-ot} is stated in Algorithms \ref{alg:DomDecIter} and \ref{alg:DomDec}. 
This formalizes the brief description given in Section \ref{sec:motivation}.
The key idea is as follows: given a current candidate coupling $\pi \in \Pi(\mu,\nu)$ and some composite partition $\partGeneric$, the objective \eqref{eq:regularized-ot} can be decomposed as
\begin{align}
\label{eq:domdec-score-decomp}
\la c, \pi \ra
+
\veps \KL(\pi \mid \mu \otimes \nu)
= \sum_{J \in \partGeneric} \left[\vphantom{\sum}\la c,\pi_J \ra + \veps \KL(\pi_J \mid \mu_J \otimes \nu)\right]
\end{align}
where $\pi_J \assign \pi \restr (X_J\times Y)$.
If one now optimizes each given restricted $\pi_J$ for $J \in \partGeneric$ subject to the temporary constraint that its composite cell marginals $\mu_J$ and $\nu_J$ remain fixed, each term in the above sum can be optimized independently and thus the for loop in Algorithm \ref{alg:DomDecIter} line \ref{codeline:cell-for-loop} can be parallelized. These temporary constraints may be added if one alternates this procedure between two staggered composite partitions $\partA$ and $\partB$.

Convergence of domain decomposition to the optimal coupling hinges on some form of connectivity property of $\partA$ and $\partB$, see \cite{BenamouPolarDomainDecomposition1994,asymptotic_domdec_arxiv,BoSch2020}.
In unbalanced optimal transport these temporary constraints cannot be added as they prevent the overall marginals of $\pi$ from changing. Thus a different parallelization strategy is required, which we develop in this article.

\begin{algorithmfloat}[hbt]
	\noindent
	\textbf{Input}: current coupling $\pi\in \measp(X\times Y)$ and partition $\partGeneric$
	
	\noindent
	\textbf{Output}: new coupling $\pi'\in \measp(X\times Y)$
	\smallskip
	
	\begin{algorithmic}[1]
		\ForAll{$J \in \partGeneric$}
		\label{codeline:cell-for-loop}
		\Comment{iterate over each composite cell}
		\State $\mu_J \leftarrow \proj_X(\pi \restr(X_{J} \times Y))$
		\label{codeline:X-marginal}
		\label{codeline:get-muJ}
		\Comment{compute $X$-marginal on cell}
		\State $\iter{\nu_J}{}\, \leftarrow \proj_Y(\pi \restr(X_{J} \times Y))$
		\label{codeline:get-composite-marginal}
		\Comment{compute $Y$-marginal on cell}
		\label{codeline:get-nuJ}
		\State $\iter{\pi_{J}}{} \leftarrow \arg\min \left\{
		\la c, \pi \ra + \veps\,\KL(\pi|\mu_J \otimes \nu) \,\middle|\, \pi \in \Pi\left(\mu_J,\iter{\nu_J}{}\right)\right\}$ 
		\label{codeline:cell-ot-problem}
		\EndFor
		\State $\pi' \leftarrow \sum_{J \in \partGeneric} \iter{\pi}{}_{J}$
	\end{algorithmic}
	\caption{\textsc{DomDecIter} \protect{\cite[Algorithm 1]{BoSch2020}}}
	\label{alg:DomDecIter}
\end{algorithmfloat}

\begin{algorithmfloat}[hbt]
	\noindent
	\textbf{Input}: an initial coupling $\piInit \in \measp(X\times Y)$
	
	\noindent
	\textbf{Output}: a sequence $(\iter{\pi}{k})_{k}$ in $\measp(X\times Y)$
	\smallskip
	
	\begin{algorithmic}[1]
		\State $\iter{\pi}{0} \leftarrow \piInit$
		\State $k \leftarrow 0$
		\Loop
		\State $k \leftarrow k+1$
		\State \algorithmicif\ ($k$ is odd)\ \algorithmicthen\ $\partk \leftarrow \partA$\ \algorithmicelse\ $\partk \leftarrow \partB$
		\Comment{select the partition}
		\State $\iter{\pi}{k} \leftarrow \text{\textsc{DomDecIter}}(\pi^{k-1}, \partk)$
		\Comment{solve all subproblems}
		\EndLoop
	\end{algorithmic}
	\caption{Domain decomposition for optimal transport \protect{\cite[Algorithm 1]{BoSch2020}}}
	\label{alg:DomDec}
\end{algorithmfloat}

For a performant implementation one should merely store $(\nu_i)_{i\in I}$ instead of the full $\pi$, since they require less memory and are sufficient to compute $\nu_J$ in line \ref{codeline:get-composite-marginal} of Algorithm \ref{alg:DomDecIter}, and the full iterates can be reconstructed efficiently from the basic cell $Y$-marginals when necessary.
For more details on the implementation of the domain decomposition algorithm we refer to \cite[Section 6]{BoSch2020}.

\section{Domain decomposition for unbalanced optimal transport}
\label{sec:unbalanced-domdec}

\subsection{Formulation of sequential and parallel algorithms}
\label{sec:formulation-sequential-parallel}
Now we want to consider domain decomposition for an unbalanced entropy regularized optimal transport problem. That is, we want to adapt Algorithm \ref{alg:DomDec} to solve \eqref{eq:unbalanced-entropic-ot}.
If one considers a splitting of the objective analogous to \eqref{eq:domdec-score-decomp} one obtains
\begin{multline}
\label{eq:domdec-score-decomp-unbalanced}
\la c, \pi \ra 
+
\varepsilon \mathrm{KL}(\pi \mid \mu\otimes\nu)
+ D_1(\proj_X\pi\mid\mu) 
+ D_2(\proj_Y\pi\mid\nu) \\
= \sum_{J \in \partGeneric} \left[\vphantom{\sum}\la c,\pi_J \ra + \veps \KL(\pi_J \mid \mu_J \otimes \nu)
+ D_1(\proj_X\pi_J\mid\mu_J) \right]
+ D_2\left( \sum_{J \in \partGeneric} \proj_Y \pi_J \middle|\nu\right).
\end{multline}
The three terms in the first sum can be treated independently, as in \eqref{eq:domdec-score-decomp}. However in the $D_2$-term all composite cells interact because typically the supports of the composite cell marginals $\nu_J = \proj_Y \pi_J$ overlap and the $Y$-marginal of $\pi$ must be allowed to change in unbalanced transport. Hence a temporary constraint to fix $\nu_J$ for each cell as in balanced transport is not admissible.

Naively, it is possible to consider the cells $J \in \partGeneric$ sequentially. Then for each cell $J$ we decompose $\pi = \pi\restr (X_J \times Y) + \pi \restr ( (X\backslash X_J) \times Y ) =: \pi_J + \pi_{-J}$, and we optimize
\begin{equation}
\label{eq:cell-primal-problem-pre}
\begin{aligned}
\min_{\pi_J\in\meas_+(X_J\times Y)}
& \la c, \pi_J+\pi_{-J} \ra 
+
\varepsilon \mathrm{KL}(\pi_J+\pi_{-J} \mid \mu\otimes\nu)
\\
&+ D_1(\proj_X[\pi_J+\pi_{-J}]\mid\mu) 
+ D_2(\proj_Y[\pi_J+\pi_{-J}]\mid\nu).
\end{aligned}
\end{equation}
The measure $\proj_Y \pi_{-J}$ appearing in the last term is precisely the \emph{background} measure $\nu_-$ that appears in Definition \ref{def:unbalanced-entropic-transport}.
For later convenience, Definition \ref{def:cell-subproblem} introduces some notation and conventions for this \emph{cell subproblem} \eqref{eq:cell-primal-problem-pre} for a fixed cell $J$.
\begin{definition}[Unbalanced cell subproblem]
	\label{def:cell-subproblem}
	For a given $\pi \in \measp(X \times Y)$ and a cell $J \in \partGeneric$ of some composite partition $\partGeneric$ we set
	\begin{equation}
	\label{eq:cell-Y-sub-marginal}
	\nu_{-J} \assign \mathcal{V}_{-J}(\pi) \qquad \tn{where} \qquad
	\mathcal{V}_{-J} : \pi \mapsto \proj_Y \pi_{-J} = \proj_Y [\pi \restr ( (X \backslash X_J) \times Y)]
	\end{equation}
	and refer to $\nu_{-J}$ as the \emph{cell $Y$-submarginal} because it is the remainder of $\proj_Y\pi$ once $\nu_J$ is subtracted. It represents our knowledge of the $Y$ marginal outside of cell $J$. The \emph{unbalanced cell subproblem} is given by
	\begin{align}
	\nonumber
	\min _{\pi_J \in \mathcal{M}_{+}\left(X_J \times Y\right)}
	E_J(\pi_J \mid \nu_{-J}),
	\quad
	\tn{with }  &E_J(\pi_J \mid \nu_{-J}) \assign
	\la c, \pi_J \ra 
	+
	\varepsilon \mathrm{KL}(\pi_J \mid \mu_J\otimes\nu)
	\\
	&+
	D_1\left(\proj_X \pi_J \mid \mu_J\right)
	+
	D_2\left(\proj_Y \pi_J+\nu_{-J} \mid \nu\right).
	\label{eq:cell-primal-problem}
	\end{align}
	The dual problem is, owing to \eqref{eq:dual-problem}, given by
	\begin{align}
	\nonumber
	\max _{\substack{\alpha \in \cont(X_J) \\ \beta \in \cont(Y)}}
	D_J(\alpha, \beta \mid \nu_{-J}),
	\ 
	\tn{with }
	D_J(&\alpha, \beta \mid \nu_{-J}) \assign 
	\,\varepsilon \int_{X_J \times Y} 1 - \exp \left(\tfrac{\alpha \oplus \beta - c}{\varepsilon}\right) \mathrm{d} (\mu_J\otimes\nu)
	\\
	&-
	\int_{X_J} \varphi_1^*(-\alpha) \mathrm{d} \mu_J
	-
	\int_Y \varphi_2^*(-\beta) \mathrm{d} \nu 
	- \la \beta , \nu_{-J} \ra.
	\label{eq:cell-dual-problem}
	\end{align}
\end{definition}

The naive adaptation of Algorithm \ref{alg:DomDecIter} for domain decomposition of unbalanced transport is then given by replacing the balanced cell problem in line \ref{codeline:cell-ot-problem} by an appropriate unbalanced counterpart \eqref{eq:cell-primal-problem-pre} of the form \eqref{eq:cell-primal-problem} (see also \eqref{eq:unbalanced-problem}) and to process the for loop sequentially.
This is formalized in Algorithm \ref{alg:UnbalancedDomDecSequential}, which is then plugged into Algorithm \ref{alg:DomDec}. Our ultimate goal in this article is to obtain a parallel version of this algorithm.

\begin{algorithmfloat}[hbt]
	\noindent
	\textbf{Input}: current coupling $\pi\in \measp(X\times Y)$ and partition $\partGeneric$
	
	\noindent
	\textbf{Output}: new coupling $\pi'\in \measp(X\times Y)$
	\smallskip
	
	\begin{algorithmic}[1]
		\State $\pi' \leftarrow \pi$
		\ForAll{$J \in \partGeneric$}
		\Comment{iterate over each composite cell (sequential!)}
		\State $\pi_{J} \leftarrow \pi' \restr(X_J \times Y)$
		\State $\iter{\nu_{-J}}{}\, \leftarrow \proj_Y (\pi' - \pi_J)$
		\label{codeline:get-composite-submarginal}
		\State $\iter{\tilde{\pi}_{J}}{}  \leftarrow \underset{\pi_J \in \measp(X_J \otimes Y)}{\argmin} 
		E_J(\pi_J \mid \nu_{-J})$
		\label{codeline:cell-uot-problem}
		\Comment{solve subproblem, see \eqref{eq:cell-primal-problem}}
		\State $\pi' \leftarrow \pi' - \pi_J + \tilde{\pi}_J$
		\Comment{update plan after every cell iteration}
		\EndFor
	\end{algorithmic}
	\caption{Sequential \textsc{DomDecIter} for unbalanced transport}
	\label{alg:UnbalancedDomDecSequential}
\end{algorithmfloat}

\begin{remark}
	\label{rk:score-non-increasing}
	Algorithm \ref{alg:UnbalancedDomDecSequential} is well defined, and can be interpreted as a block-coordinate descent algorithm for the problem \eqref{eq:unbalanced-problem}, where in each iteration we optimize over the entries within a composite cell. Hence, the score is non-increasing.
\end{remark}

Convergence of Algorithm \ref{alg:UnbalancedDomDecSequential} will be proved in Section \ref{sec:convergence-sequential}. 
Deriving a parallel version is far from straightforward, since when several cells are updated simultaneously, the information in $\nu_{-J}$ is no longer reliable. In particular, several subproblems may compete to create (or destroy) mass at the same positions in $Y$, which may appear to be optimal from each problems' perspective but is actually globally detrimental. If unaddressed, these dynamics can grow into a feedback loop with diverging iterates, as exemplified later in Section \ref{sec:multiscale}.

Fortunately, it is possible to combine independent local updates with guaranteed global score decrement and convergence. Algorithm \ref{alg:UnbalancedDomDecParallel} gives a general method for parallelizing unbalanced domain decomposition updates. The subproblems are solved in parallel, using the previous iterate for computing $\nu_{-J}$. Subsequently, each cell is updated with a convex combination of the previous and the new coupling for that cell.
The weights are determined by an auxiliary function \GetWeights, Algorithm \ref{alg:UnbalancedDomDecParallel}, line \ref{codeline:getWeights}.
Convergence of the algorithm now hinges on properties of this function.
Naive parallelization of the loop in Algorithm \ref{alg:UnbalancedDomDecSequential} would correspond to the function
\begin{equation}
\label{eq:GWFast}
\tn{\GetWeightsSub{greedy}} : ((\pi_J)_J, (\tilde{\pi}_J)_J) \mapsto (1,\ldots,1)
\end{equation}
which always returns the greedy choice $\theta_J=1$ for all $J$. While this may lead to diverging iterates, by convexity of the objective, on each cell $J$, $\tilde{\pi}_J$ lies in a direction of $\pi_J$ where the objective locally decreases. Therefore, if one moves with sufficiently small step-size from the previous iterates towards the new iterates, one can ensure that the global objective decreases and the algorithm converges. However, this may yield small, impractical step sizes on problems with many cells. Fortunately, other, more adaptive choices with larger step sizes can also be identified.
In Section \ref{sec:domdec-parallel} we give sufficient conditions on $\GetWeights$ for convergence, as well as several examples satisfying these conditions.

\begin{algorithmfloat}[hbt]
	\noindent
	\textbf{Input}: current coupling $\pi\in \measp(X\times Y)$ and partition $\partGeneric$
	
	\noindent
	\textbf{Output}: new coupling $\pi'\in \measp(X\times Y)$
	\smallskip
	
	\begin{algorithmic}[1]
		\ForAll{$J \in \partGeneric$}
		\Comment{iterate over each composite cell}
		\State $\pi_{J} \leftarrow \pi \restr(X_J \times Y)$
		\State $\iter{\nu_{-J}}{}\, \leftarrow \proj_Y (\pi - \pi_J)$
		\State $\iter{\tilde{\pi}_{J}}{}  \leftarrow \underset{\pi_J \in \measp(X_J \otimes Y)}{\argmin} 
		E_J(\pi_J \mid \nu_{-J})$	
		\Comment{solve subproblem}
		\EndFor
		\State $(\theta_J)_J\leftarrow \text{\GetWeights}((\pi_J)_J, (\tilde{\pi}_J)_J)$
		\label{codeline:getWeights}
		\State $\pi' \leftarrow  \underset{{J\in \partGeneric}}{\sum} (1- \theta_J) \pi_J + \theta_J\tilde{\pi}_{J}$	
		\Comment{convex combination of old and new cell plans}
	\end{algorithmic}
	\caption{Parallel \textsc{DomDecIter} for unbalanced transport. }
	\label{alg:UnbalancedDomDecParallel}
\end{algorithmfloat}

\subsection{Properties of the cell subproblem}
\label{sec:properties-cell-subproblem}
In this section we establish properties of the unbalanced cell problem (Definition \ref{def:cell-subproblem}) that are required for the convergence analysis of the sequential and parallel domain decomposition algorithms.

\begin{assumption}\label{ass:divergence}
	We focus on entropy functionals $\varphi$ (in the sense of Definition \ref{def:divergence}) with the following properties:
	\begin{enumerate}
		\item $\varphi$ is non-negative. \label{item:entropy-bounded-below}
		\item $\varphi$ is finite on $[0, \infty)$. \label{item:entropy-zero-finite}
		\item $\varphi$ is continuously differentiable on the interior of its domain.
	\end{enumerate}
\end{assumption}
The first two assumption ensure that the infimum in \eqref{eq:cell-primal-problem} is not $-\infty$ or $+\infty$, respectively. The third assumption is, in virtue of Theorem \ref{theorem:existence-global-solution}, sufficient for uniqueness of the optimal dual potentials, which simplifies their convergence analysis.

Moreover, in the following we will assume that $X$ and $Y$ are finite. This is sufficient to show that the solution to the cell problem \eqref{eq:cell-primal-problem} is stable with respect to the input data (Lemma \ref{lemma:continuity-solve-map}). For general $Y$ and $\nu$ the cell problem \eqref{eq:cell-primal-problem} is in general not stable, as illustrated in Remark \ref{remark:continuousY}.
For simplicity we also assume that $\mu$ and $\nu$ have full support. Otherwise we can simply drop points with zero mass from $X$ and $Y$.
\begin{assumption}\label{ass:finite}
	The spaces $X$ and $Y$ are finite. $\mu$ and $\nu$ have full support.
\end{assumption}

The following properties of $\varphi^*$ directly follow from Assumption \ref{ass:divergence}. 
The proofs can be found in any convex analysis textbook, for example \cite[Chapter X]{convex-analysis-textbook}.

\begin{lemma}[Properties of $\varphi^*$]
	\label{lemma:properties-conjugate}
	For $\varphi$ satisfying Assumption \ref{ass:divergence}, $\varphi^*$ has the following properties:
	\begin{enumerate}
		\item $\varphi^*(0)$ is finite. 
		\item $\varphi^*$ is bounded from below by $-\varphi(0)$.
		\item $\varphi^*$ is superlinear, i.e., $\lim_{z\rightarrow \infty}\varphi^*(z) /z = + \infty$.
	\end{enumerate}
\end{lemma}

The following result gives existence and uniqueness of optimizers for finite spaces.
\begin{lemma}[Properties of cell optimizers on finite spaces]
	\label{lemma:uniform-bound-duals}

	Assume that Assumptions \ref{ass:divergence} and \ref{ass:finite} hold.
	Then:
	\begin{enumerate}
		\item \label{item:existence-primal} There exists a unique cell primal optimizer $\pi_J$ for \eqref{eq:cell-primal-problem}.
		
		\item \label{item:existence-dual} There exist unique dual optimizers $(\alpha_J, \beta_J)$ for \eqref{eq:cell-dual-problem}. The primal and dual optimizers satisfy 
		\begin{equation}
		\label{eq:diagonal-scaling}
		\pi_J 
		=
		\exp \left(
		\tfrac{\alpha_J \oplus \beta_J - c}{\veps} 
		\right)\mu_J \otimes \nu,
		\  
		\tn{and}
		\ 
		\begin{cases}
		-\alpha_J(x) = \varphi_1'\left(
		\RadNik{\proj_X \pi_J}{ \mu_J}(x)
		\right)
		\\
		-\beta_J(y) = \varphi_2'\left(
		\RadNik{(\proj_Y \pi_J+\nu_{-J})}{\nu}(y)
		\right) 
		\end{cases}
		\end{equation}
		for all $(x,y) \in X \times Y$.
		\item \label{item:bound-dual} Let $M$ be such that $0 \le \nu_{-J} \le M \nu$. Then there exists a positive constant $C$ (depending only on $\mu_J$, $\nu$, $c$, $\veps$ and $M$) such that 
		\begin{equation*}
		\|\alpha_J\|_{\infty} \leq C 
		\quad \text { and } \quad 
		\|\beta_J\|_{\infty} \leq C.
		\end{equation*}
	\end{enumerate}
\end{lemma}

\begin{proof}
	
	With finite support and the assumption that $\varphi$ is finite on $[0,\infty)$, the minimal value of \eqref{eq:cell-primal-problem} is finite (for instance, the zero measure $\pi=0$ yields a finite objective).
	Therefore Theorem \ref{theorem:existence-global-solution} can be applied. This yields Statement \ref{item:existence-primal} and the optimality condition \eqref{eq:diagonal-scaling} (if dual maximizers exist) where we use that since by Assumption \ref{ass:divergence} $\varphi_1$ and $\varphi_2$ are differentiable, the subdifferential in the optimality condition of Theorem \ref{theorem:existence-global-solution} only contains one element, and that $\mu$-almost every $x$ is every $x$ on a finite space with fully supported $\mu$ (and likewise for $\nu$).
	
	By construction the dual objective $D(\cdot, \cdot | \nu_{-J})$, \eqref{eq:cell-dual-problem}, is upper semi-continuous, by finiteness of $\varphi_i^*(0)$ there are feasible dual candidates.
	To establish existence of dual maximizers we need to show that the superlevel sets of $D(\cdot, \cdot | \nu_{-J})$ are bounded.
	Given existence, the left side of \eqref{eq:diagonal-scaling} then shows strict positivity of the primal minimizer $\pi_J$ on the support of $\mu_J \otimes \nu$. The right side then gives uniqueness of the dual maximizers.
	
	Boundedness of the superlevel sets of $D(\cdot, \cdot | \nu_{-J})$ is a rather simple consequence of the superlinearity of $\varphi_i^*$. We still describe the argument in detail, since it provides the bounds for Statement \ref{item:bound-dual}.
	
	Let $L \in \R$ and assume that $(\alpha,\beta)$ are dual candidates in the corresponding superlevel set, i.e.~$L \leq D(\alpha, \beta \mid \nu_{-J})$. Then for any $(x,y) \in X_J \times Y$,
	\begin{align*}
	\nonumber
	L &\le D(\alpha, \beta \mid \nu_{-J}) 
	\le 
	\veps \|\mu_J \otimes \nu \|
	- \veps \exp \left(\frac{\alpha(x) + \beta(y) - c(x,y)}{\varepsilon}\right) 
	\mu_J(\{x\})\nu(\{y\})
	\\
	\nonumber
	&
	\ \ - \varphi_1^*(-\alpha(x)) \cdot \mu_J(\{x\})
	+
	\varphi_1(0)\mu_J(X_J \setminus \{x\})
	\\
	&
	\ \ - \varphi_2^*(-\beta(y)) \cdot \nu (\{y\})
	- \beta(y) \nu_{-J}(\{y\})
	+ \int_{Y\backslash\{y\}}
	\varphi_2(\RadNik{\nu_{-J}}{\nu}(y))
	\diff \nu(y)
	,
	\end{align*}
	where we use Lemma \ref{lemma:properties-conjugate} to bound the contribution of the remaining terms. 
	Further, using that $\RadNik{\nu_{-J}}{\nu}(y)$ is non-negative and bounded by some $M < \infty$ (since $Y$ is finite, $\nu$ has full support, and all masses assigned to single points are finite), and that convex functions assume their maximum on the boundary of their domain, the integrand of the last term must be bounded by the maximum of $\varphi_2(0)$ and $\varphi_2(M)$. 
	Collecting the terms $\veps \| \mu_J\otimes \nu\|$, $\varphi_1(0)\mu_J(X_J)$ and $	\max \left\{
	\varphi_2(0), \varphi_2(M)
	\right\}
	\nu(Y)$ in a constant $C$ we obtain:
	\begin{align*}
	L&\le 
	C
	- \veps \exp \left(\frac{\alpha(x) + \beta(y)}{\varepsilon}\right) 
	\exp \left(\frac{-c(x,y)}{\varepsilon}\right) 
	\mu_J(\{x\})\nu(\{y\})
	\\
	\nonumber
	&
	\ \ - \varphi_1^*(-\alpha(x)) \cdot \mu_J(\{x\})
	- 
	\left[
	\varphi_2^*(-\beta(y))  
	+ \beta(y)\RadNik{\nu_{-J}(y)}{\nu}
	\right] \nu (\{y\})
	\end{align*}
	
	We now show how the lower bounds and superlinearity of the entropy functions allow to bound $\alpha$ and $\beta$ from below. Let us outline the argument in detail for $\beta$: assume that $\beta(y)$ is negative and use the bound from below of the exponential term and of $\varphi_1^*$:
	\begin{align*}
	L&\le 
	C
	+
	\varphi_1(0) \mu_J(\{x\})
	-
	\left[
	\varphi_2^*(-\beta(y))  
	+ \beta(y)\RadNik{\nu_{-J}(y)}{\nu}
	\right] \nu (\{y\})
	\\
	&\le 
	C
	+
	\varphi_1(0) \mu_J(\{x\})
	-
	\left[
	\varphi_2^*(-\beta(y))  
	+ M\beta(y)
	\right] \nu (\{y\}).	
	\end{align*}
	Then, the superlinearity of $\varphi_2^*$ implies a lower bound on $\beta(y)$ that does not depend on $\nu_{-J}$ nor $\alpha(x)$.
	Likewise, using the uniform lower bound for the $\varphi_2^*$ term yields a lower bound on $\alpha(x)$, that does not depend on $\nu_{-J}$ and $\beta(y)$.
	Finally, by the superlinearity of $\exp(\cdot)$, the total lower boundedness of $\varphi^*_i$, once more the bound $0 \leq \nu_{-J} \leq M \cdot \nu$, and by the individual lower bounds on $\alpha(x)$ and $\beta(y)$, we obtain similar upper bounds on $\alpha(x)$ and $\beta(y)$.
	Taking now the extrema over these bounds with respect to $(x,y) \in X \times Y$, we obtain existence of some $C<\infty$ that only depends on $\mu$, $\nu$, $c$, $\veps$, $M$, and $L$ but not on $\nu_{-J}$, such that $\|\alpha\|_\infty \leq C$, and $\|\beta\|_\infty \leq C$.	
	This implies boundedness of the superlevel sets.
	
	Finally, for Statement \ref{item:bound-dual}, note that setting $(\alpha,\beta)=(0,0)$ in \eqref{eq:cell-dual-problem} yields a value for $L$ such that the corresponding superlevel set is non-empty (and therefore contains the maximizer), and this $L$ is in turn easy to bound from below in terms that only depend on $\mu$, $\nu$, $c$, $\veps$, and $M$.
\end{proof}

\begin{lemma}[Continuity of the solution map]
	\label{lemma:continuity-solve-map}
	Let $S_J$ be the mapping from a global transport plan to the solution of the domain decomposition subproblem on cell $J$. More precisely, recall the $Y$-submarginal map $\mathcal{V}_{-J} : \pi \mapsto \nu_{-J} \assign  \proj_Y \pi \restr ( (X \backslash X_J) \times Y)$  from \eqref{eq:cell-Y-sub-marginal} and let
	\begin{align}
	\label{eq:solve-map}
	S_J : \ \measp(X\times Y) &\rightarrow \measp(X_J\times Y) &
	\pi &\mapsto \argmin_{\hat{\pi}_J \in \measp(X_J \otimes Y)} 
	E_J(\hat{\pi}_J \mid \mathcal{V}_{-J} (\pi) ).
	\end{align}
	Under Assumptions \ref{ass:divergence} and \ref{ass:finite} the map $S_J$ is continuous on the set of plans $\pi$ where $E_J(\pi \restr (X_J \times Y)|\mathcal{V}_{-J}(\pi))<\infty$.
\end{lemma}

\begin{proof}[Proof]
	Let $(\pi^n)_n$ be a sequence converging to some limit $\pi$, such that $E_J(\pi^n \restr (X_J \times Y)|\nu^n_{-J})$ is finite on the whole sequence and similarly for the limit, where as in \eqref{eq:cell-Y-sub-marginal}, we define the $Y$-submarginals:
	\begin{align}
	\label{eq:submarginals}
		\nu_{-J}^n & \assign \mathcal{V}_{-J}(\pi^n), &
		\nu_{-J} & \assign \mathcal{V}_{-J}(\pi).
	\end{align}
	Recall that $\mathcal{V}_{-J}$ is continuous, since it is a combination of restriction (which is continuous in the finite-dimensional setting) and projection (which is always continuous).
	
	Let $\pi^n_J$ (resp. $\pi_J$) denote the minimizer of $E(\cdot \mid \nu_{-J}^n)$, (resp. $E(\cdot \mid \nu_{-J})$) which exists by virtue of Lemma \ref{lemma:uniform-bound-duals}. 
	Since $E_J$ is jointly continuous at $(\pi_J, \nu_{-J})$, for sufficiently large $n$ it holds: 
	\begin{equation}
		\label{eq:comparison-compactness-pinJ}
		E_J(\pi^n_J \mid \nu_{-J}^n)
		\le 
		E_J(\pi_J\mid \nu_{-J}^n)
		\le 
		E_J(\pi_J \mid \nu_{-J}) + 1,
	\end{equation} 
	where in the first inequality we used that $\pi^n_J$ is the minimizer for $\nu_{-J}^n$. 
	Neglecting the marginal penalizations in \eqref{eq:cell-primal-problem} (which by assumption are positive), we obtain 
	\begin{equation}
		\la c, \pi^n_J \ra 
		+
		\varepsilon \mathrm{KL}(\pi^n_J \mid \mu_J\otimes\nu)
		\le 
		E_J(\pi_J \mid \nu_{-J}) + 1,
	\end{equation}
	which in conjunction with the super-linearity of the $\KL$ divergence implies pre-compactness of the sequence $(\pi^n_J)_n$.
	Taking a subsequence converging to some plan $\pi_J^*$, and taking the limit $n\to\infty$ in \eqref{eq:comparison-compactness-pinJ} finally yields 
	\begin{equation}
		E_J(\pi^*_J \mid \nu_{-J}) \le E_J(\pi_J \mid \nu_{-J}).
	\end{equation}
	
	Since by Lemma \ref{lemma:uniform-bound-duals} the minimizer of $E_J(\cdot \mid \nu_{-J}$) is unique, we conclude that $\pi^*_J = \pi_J$ and that $\pi^n_J \to \pi_J$ along the original sequence.
	
\end{proof}

\begin{remark}
	\label{remark:continuousY}
	The solution map \eqref{eq:solve-map} is not continuous for general $Y$. When $\pi^n$ is a converging sequence such that the sequence $\nu_{-J}^n$ oscillates strongly, the limit of the cell optimizers $S_J(\pi^n)$ might be suboptimal for the limit problem $E(\cdot \mid \nu_{-J})$, when $\nu_{-J}$ is the weak* limit of $(\nu_{-J})^n$. As an example, consider the following setting: 
	\begin{equation}
	\begin{gathered}
	X = \{0,1\}, \quad X_J = \{0\},
	\qquad 
	Y = [0,1],
	\qquad
	c \equiv 0,
	\qquad
	\veps = 0,
	\qquad
	D_1 = D_2 = \KL,
	\\
	\mu_J = M\delta_0\ \tn{ with $0 < M < 1/2$}, 
	\qquad 
	\nu = \Lebesgue\restr[0,1].
	\end{gathered}
	\end{equation}
	The example is easier with $\veps=0$ but, as will be seen, the argument can be generalized to $\veps>0$.
	Since $X_{J}$ and $X \setminus X_{J}$ are just single points, $\pi^n_J$ and $\pi^n_{-J}$ can be identified with their $Y$-marginals $\nu_J^n$ and $\nu_{-J}^n$. For $\nu_J^n$ we pick an oscillating square wave pattern with frequency $n$:
	$$ 
	\nonumber
		\nu_J^n \assign
		H(\sin(2\pi n))\nu, \quad \text{with } H(s) = +1 \text{ if $s < 0$, otherwise $H(s)=0$,}
	$$
	and set $\nu^n_{-J} \assign \nu-\nu^n_J$ as the complementary square wave (see Figure \ref{fig:counterexample_1}).
	The weak* limit of $\nu_{-J}^n$ is clearly $\nu_{-J} = \tfrac{1}{2} \nu$. The optimizer for $E_J(\cdot \mid \nu_{-J}^n)$ then must feature a similar oscillation, being of the form $\pi_J^n = \delta_0 \otimes (a \nu_{J}^n + b\nu_{-J}^n)$, see Figure \ref{fig:counterexample_2}; one can show that the optimal parameters are $(a, b) = (\sqrt{2M},0)$.
	
	Indeed, rewriting the primal problem \eqref{eq:cell-primal-problem} using that $\nu_{-J}$ has piecewise constant density yields 
	\begin{equation}
		\min_{a,b\ge 0}
		M \varphi_{\KL}\left(\frac{1}{M}\frac{a+b}{2}\right) 
		+
		\frac{1}{2} (\varphi_{\KL}(a) + \varphi_{\KL}(b+1)).
	\end{equation}
	Finally, by differentiating and applying the Karush-Kuhn-Tucker conditions we find that, for $0<M<1/2$, the only solution is attained at $a = \sqrt{2M}$, $b = 0$.
	
	The weak* limit of these couplings is then given by $\pi_J = \sqrt{M/2}\, \delta_0 \otimes \nu$, which does not constitute the optimizer for the limit problem; instead, by employing again the Karush-Kuhn-Tucker conditions, the optimizer for the limit problem can be shown to be $\pi'_J = (\sqrt{1/16 + M}-1/4)\, \delta_0 \otimes \nu$. 
	
		\begin{figure}[h]
		\centering
		\begin{subfigure}[b]{0.4\textwidth}
			\includegraphics[width=\linewidth]{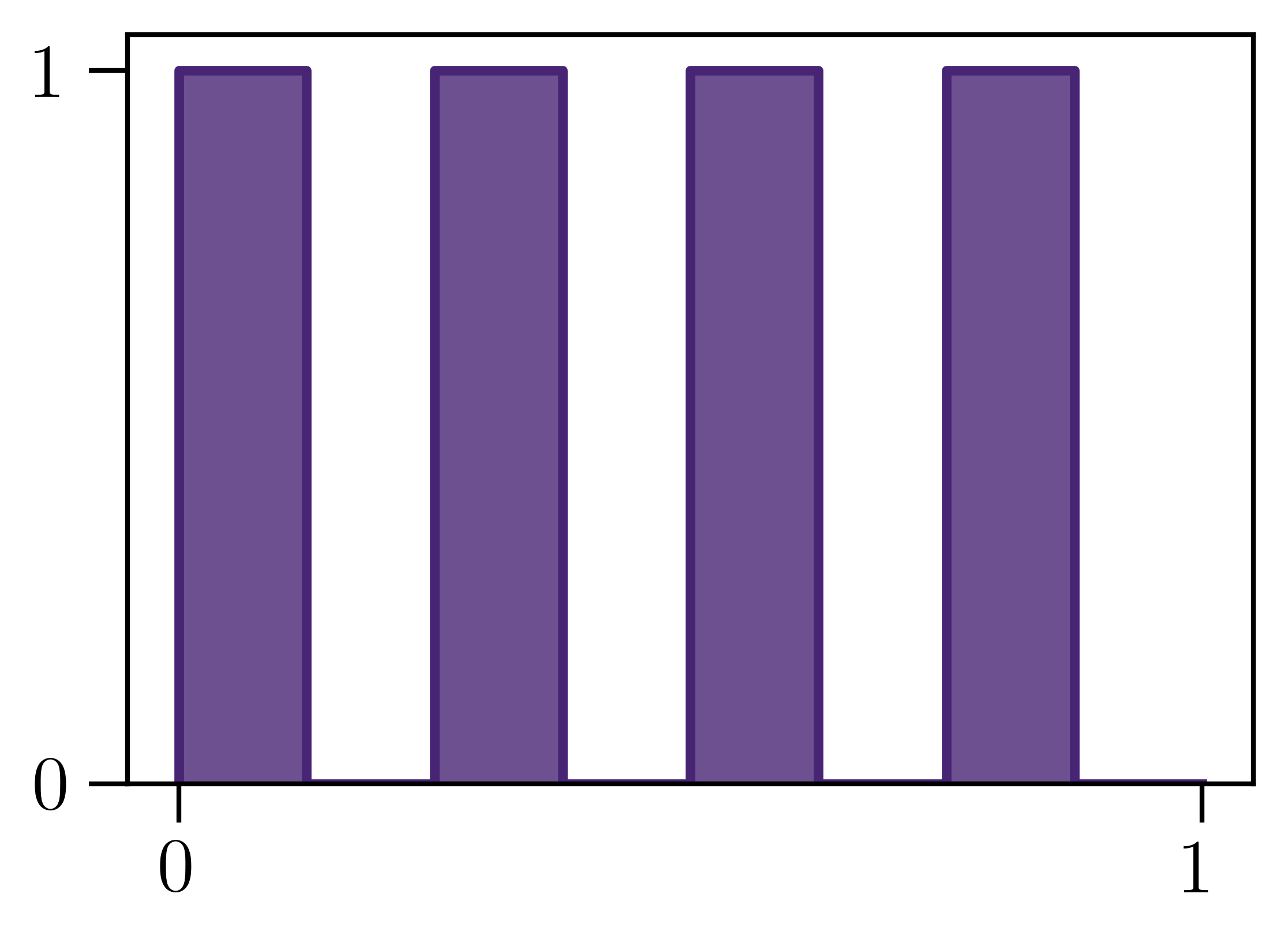}
			\caption{$\nu^{n}_{-J}$}
			\label{fig:counterexample_1}
		\end{subfigure}
		\hspace{15mm}
		\begin{subfigure}[b]{0.4\textwidth}
			\includegraphics[width=\linewidth]{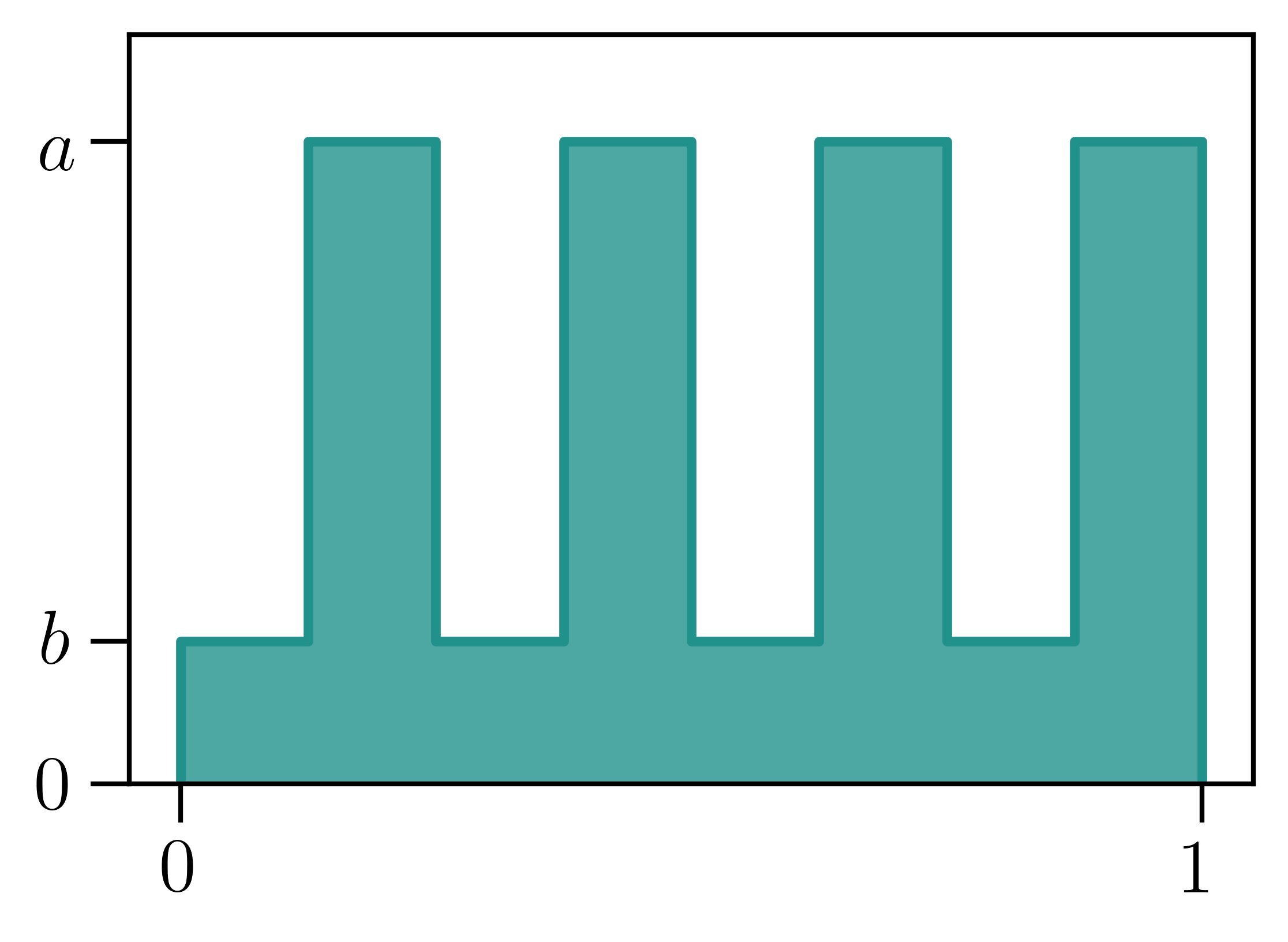}
			\caption{$\pi_J^n$}
			\label{fig:counterexample_2}
		\end{subfigure}
		\caption{Illustration of $\nu^n_{-J}$ and $\pi^{n}_J$ in the counterexample of Remark \ref{remark:continuousY}.}
		\label{fig:counterexample}
	\end{figure}
	
	This lack of stability is a result of the strict lower-semicontinuity of the $\KL$ soft-marginal in the continuous setting. 
	Intuitively, this kind of wildly oscillating sequences seems unlikely to arise in iterations of the domain decomposition algorithm due to the monotonous character of the score. Unfortunately, previous attempts to bound oscillations in a similar context only covered some simple special cases \cite[Section 4.3]{asymptotic_domdec_arxiv}.

\end{remark}

\subsection{Convergence of the sequential algorithm for finitely supported measures}
\label{sec:convergence-sequential}

\begin{proposition}
	\label{prop:convergence-sequential-algorithm}
	Let Assumptions \ref{ass:divergence} and \ref{ass:finite} hold, let $I$ be a basic partition and $(\partA$, $\partB)$ two composite partitions in the sense of Definition \ref{def:partitions}, let $\piInit$ feasible for problem \eqref{eq:unbalanced-entropic-ot}. Then Alg. \ref{alg:DomDec} with iteration subroutine given by Alg. \ref{alg:UnbalancedDomDecSequential} converges to the unique globally optimal minimizer of problem \eqref{eq:unbalanced-entropic-ot}. 
\end{proposition}

\begin{proof}
	The iterates of Algorithm \ref{alg:UnbalancedDomDecSequential} have non-increasing score in \eqref{eq:unbalanced-entropic-ot} by virtue of line \ref{codeline:cell-uot-problem} (see also Remark \ref{rk:score-non-increasing}). Therefore, by using finiteness of $X$ and $Y$, we can bound the density of $\pi^k$ with respect to $\mu \otimes \nu$ uniformly in $k$ by using the score of the initial plan. This means, there exists $C$ such that 
	\begin{equation}
	0 \le \RadNikD{\pi^k}{(\mu\otimes \nu)} \le C
	\quad
	\tn{for all } k\ge 1.
	\end{equation}
	
	This uniform bound on the entries of $\pi^k$ allows to extract cluster points. Since the global score \eqref{eq:unbalanced-entropic-ot} is (in the finite support setting) a continuous function on its domain, all cluster points share the same score.

	Consider the \emph{iteration maps} $\iterA$ and $\iterB$, which apply an $A$ (resp.~$B$) iteration to a coupling $\pi$ (i.e.~they map $\pi^k$ to $\pi^{k+1}$). They are a composition of restriction, projection, solution maps (cf.~Lemma \ref{lemma:continuity-solve-map}), and a sum of measures, all of which are continuous in our setting. 
	Now choose a cluster point $\hat{\pi}$ of the iterates. By continuity of the iteration maps, $\iterB(\iterA(\hat{\pi}))$ is also a cluster point, therefore it has the same score as $\hat{\pi}$. By strict convexity of the cell problems, this means that $\hat{\pi}$ does not change when applying $\iterA$ or $\iterB$ (otherwise the score would have decreased). Hence, $\hat{\pi}$ is optimal for each cell problem, which by item \ref{item:existence-dual} of Lemma \ref{lemma:uniform-bound-duals} means that it features the diagonal scaling form
	\begin{equation}
	\label{eq:cluster-cell-plans}
	\hat{\pi}_J 
	= 
	\exp\left(\frac{\hat{\alpha}_J \oplus \hat{\beta}_J - c}{\veps}\right) \mu_J \otimes \nu
	\quad
	\tn{for all }
	J \in \partA \cup \partB,
	\end{equation}
	where the cell duals satisfy
	\begin{equation}
	\label{eq:cluster-cell-duals}
	\begin{aligned}
	&-\hat{\alpha}_J(x) = \varphi'_1\left(\RadNik{\proj_X \hat{\pi}_J}{\mu_J}(x)\right),
	\ \text{and}
	\\
	&-\hat{\beta}_J(y) 
	= 
	\varphi'_2 \left( \RadNik{\proj_Y \hat{\pi}_J}{\nu}(y)
	+ \RadNik{\hat{\nu}_{-J}}{\nu}(y)\right)
	=
	\varphi'_2 \left( \RadNik{\proj_Y \hat{\pi}}{\nu}(y)
	\right).
	\end{aligned}
	\end{equation}
	Here $\hat{\nu}_{-J}$ is the $Y$-submarginal associated with $\hat{\pi}$, analogous to \eqref{eq:cell-Y-sub-marginal}.
	The latter condition implies that $\hat{\beta}_J$ is actually identical on all cell problems, so we can drop the index $J$ and rename it to $\hat{\beta}$. 
	One can also construct a global dual potential $\alpha$ on $X$ by choosing a partition $\partGeneric$ and, for every $x\in X$, defining $\hat{\alpha}(x) \assign \hat{\alpha}_J(x)$, where $J$ the only composite cell such that $x\in X_J$. In view of this, \eqref{eq:cluster-cell-plans} and \eqref{eq:cluster-cell-duals} become
	\begin{equation}
	\nonumber
	\hat{\pi} 
	= 
	\exp\left(\frac{\hat{\alpha} \oplus \hat{\beta} - c}{\veps}\right) \mu \otimes \nu,
	\quad
	\tn{ with }
	\quad
	\begin{cases}
	-\hat{\alpha}(x) = \varphi_1'\left(
	\RadNikD{\proj_X \hat{\pi}}{ \mu}(x)
	\right) 
	\tn{ for all } x \in X,
	\\[2mm]
	-\hat{\beta}(y) =  \varphi'_2\left(
	\RadNikD{\proj_Y \hat{\pi}}{ \nu}(y)
	\right) 
	\tn{ for all } y \in Y,
	\end{cases}
	\end{equation}
	and by applying Theorem \ref{theorem:existence-global-solution} we conclude optimality of $\hat{\pi}$ for the primal problem \eqref{eq:unbalanced-problem} and $(\hat{\alpha}, \hat{\beta})$ for the dual problem \eqref{eq:dual-problem}.
	This argument can be made for any cluster point $\hat{\pi}$, which therefore all must coincide, and thus the whole sequence of iterates converges to this limit.
\end{proof}

\begin{remark}[Unbalanced domain decomposition needs just one partition]
	\label{rk:only-one-partition-needed}
	In domain decomposition for balanced optimal transport, the two partitions $\partA$ and $\partB$ are required to satisfy a certain joint connectivity property to guarantee convergence (see e.g. \cite[Definition 4.10]{BoSch2020}). 
	No such property is required for the unbalanced variant, since the $Y$-submarginal $\nu_{-J}$  allows communication between the different subproblems, ensuring the consistency of the $Y$-cell duals upon convergence (cf.~\eqref{eq:cluster-cell-duals}).
	However, numerically we still observe a benefit of using two staggered partitions as in \cite{BoSch2020}, see Section \ref{sec:one-partition} for more details.
\end{remark}

\subsection{Parallelization and convergence of parallel algorithm}
\label{sec:domdec-parallel}

One of the key advantages of balanced domain decomposition is the partitioning of the transport problem into a collection of subproblems that can be solved in parallel. 
However, as anticipated in Section \ref{sec:formulation-sequential-parallel}, parallelizing unbalanced domain decomposition is more challenging, since the subproblems are linked by the soft $Y$-marginal penalty.

In this Section we show how to overcome this issue. 
We will prove convergence of Algorithm \ref{alg:UnbalancedDomDecParallel} to the globally optimal coupling, assuming that the function \GetWeights guarantees decrement on each update (if decrement is possible for the sequential algorithm) and that it enjoys a certain regularity, for which we offer several variants.

Proposition \ref{prop:convergence-parallel-domdec} 
gives a sufficient condition for Algorithm \ref{alg:UnbalancedDomDecParallel} to be convergent, namely, that $\GetWeights$ is continuous, and that the the score of the new plan $\pi'$ resulting from applying Algorithm \ref{alg:UnbalancedDomDecParallel} to $\pi$ is always lower than that of $\pi$ (unless $\pi$ is already optimal)
. However, continuity is too restrictive for several interesting choices of $\GetWeights$. Therefore, Proposition 
\ref{prop:convergence-parallel-domdec-extended} proves convergence with an alternative comparison argument: if the considered \GetWeights performs consistently better than a second \GetWeightsSub{0}, and the latter satisfies the assumptions of Proposition \ref{prop:convergence-parallel-domdec}, the former also yields a convergent algorithm.
Finally, Proposition \ref{prop:converging-getweights} outlines several possible choices of \GetWeights satisfying the required properties. These choices will be compared numerically in Section \ref{sec:single-scale}.

\begin{proposition}
	\label{prop:convergence-parallel-domdec}
	Assume that the function \GetWeights (Alg.~\ref{alg:UnbalancedDomDecParallel}, line \ref{codeline:getWeights})
	\begin{enumerate}
		\item \label{item:assumption-continuity} is continuous, and
		\item \label{item:parallel-decrement}
		for every initial plan $\pi \in \measp(X \times Y)$, partition $\partGeneric$, current cell plans $(\pi_J)_{J\in \partGeneric}$ and new cell plans  $(\tilde{\pi}_J)_{J\in \partGeneric} \assign (S_J(\pi))_{J\in \partGeneric}$, the updated plan
		\begin{equation}
		\nonumber
		\bar{\pi}
		\assign  
		\underset{{J\in \partGeneric}}{\sum} (1- \theta_J) \pi_J + \theta_J\tilde{\pi}_{J}, 
		\quad
		\text{with } (\theta_J)_J = \text{\GetWeights}((\pi_J)_J, (\tilde{\pi}_J)_J),
		\end{equation}
		has a global score in \eqref{eq:unbalanced-entropic-ot} that is strictly lower than that of  $\pi$, unless $\pi_J= \tilde{\pi}_J$ for all $J\in \partGeneric$ (in which case $\pi$ is already optimal, see proof of Prop. \ref{prop:convergence-sequential-algorithm}).
	\end{enumerate}
	Then, under the same conditions of Proposition \ref{prop:convergence-sequential-algorithm},  Alg.~\ref{alg:DomDec} with iteration subroutine given by Alg.~\ref{alg:UnbalancedDomDecParallel} converges to the globally optimal solution of problem \eqref{eq:unbalanced-problem}. 
\end{proposition}

\begin{proof} 
	First note that item \ref{item:parallel-decrement} grants that the sequence of plans $(\pi^k)_k$ has a non-increasing score. Thus, analogously to the proof of Proposition \ref{prop:convergence-sequential-algorithm}, we can uniformly bound the density of $(\pi^k)_k$, extract cluster points, and show that all cluster points share the same score. 
	
	The next step is to show continuity of the iterations maps $\mathcal{F}_A$ and $\mathcal{F}_B$. Now they are a composition of restriction, projection, solution maps, the function $\GetWeights$, and sum of measures. All of these functions are continuous in our setting ($\GetWeights$ by item \ref{item:assumption-continuity}), so $\mathcal{F}_A$ and $\mathcal{F}_B$ are continuous as well. As in the proof of Proposition \ref{prop:convergence-sequential-algorithm} this continuity implies that, for any cluster point of the iterates $\hat{\pi}$, $\mathcal{F}_A(\mathcal{F}_B(\hat{\pi}))$ is again a cluster point. This in turn implies that $\mathcal{F}_A(\mathcal{F}_B(\hat{\pi}))$ and $\hat{\pi}$ have the same score, which by item \ref{item:parallel-decrement} can only happen if all cell subproblems are locally optimal.
	
	The proof then concludes as in Proposition \ref{prop:convergence-sequential-algorithm}: local optimality implies that cell couplings are diagonal scalings of their respective cell duals; $Y$-cell duals are all identical and $X$-cell duals can be `stitched' to form a global $X$ dual. The cluster point $\hat{\pi}$ is thus a diagonal scaling of the resulting global duals, which satisfy the global optimality conditions, and therefore also the limit of the whole sequence. 
\end{proof}

The continuity condition required by Proposition \ref{prop:convergence-parallel-domdec} is rather restrictive and does not allow many interesting choices for the \GetWeights\ function. In the next proposition we show that we can also use non-continuous \GetWeights\ functions, as long as they perform better (or equal) than some continuous one.

\begin{proposition}
	\label{prop:convergence-parallel-domdec-extended}
	Denote by $E$ the objective function of the regularized unbalanced minimization problem \eqref{eq:unbalanced-entropic-ot}.
	Let \GetWeightsSub{0} be a function satisfying the assumptions of Proposition \ref{prop:convergence-parallel-domdec}, and \GetWeights a different function performing better than or as well as the former, i.e., such that for every given plan $\pi \in \measp(X \times Y)$, partition $\partGeneric$, current cell plans $(\pi_J)_{J\in \partGeneric}$ and new cell plans  $(\tilde{\pi}_J)_{J\in \partGeneric} \assign (S_J(\pi))_{J\in \partGeneric}$, it holds $
	E(\bar{\pi}) \le E(\bar{\pi}_0)$, with
	\begin{equation}
	\begin{cases}
	\bar{\pi}
	\assign  
	\underset{{J\in \partGeneric}}{\sum} (1- \theta_J) \pi_J + \theta_J\tilde{\pi}_{J},
	\ 
	(\theta_J)_J = \text{\GetWeights}((\pi_J)_J, (\tilde{\pi}_J)_J),\\
	\bar{\pi}_0
	\assign  
	\underset{{J\in \partGeneric}}{\sum} (1- \theta^0_J) \pi_J + \theta^0_J\tilde{\pi}_{J},
	\ 
	(\theta_J^0)_J = \text{\GetWeightsSub{0}}((\pi_J)_J, (\tilde{\pi}_J)_J).
	\end{cases} 
	\end{equation}
	Then, under the same conditions as in Proposition \ref{prop:convergence-sequential-algorithm}, Algorithm \ref{alg:UnbalancedDomDecParallel} yields a convergent domain decomposition algorithm. 
\end{proposition}

\begin{proof}
	Once again the iterates have a non-increasing score by assumption, so we can bound the density of the iterates $\pi^k$ uniformly in $k$, extract cluster points, and show that all cluster points share the same score. 
	
	Now let $\mathcal{F}_A$ and $\mathcal{F}_B$ be the maps that perform an $A$ and $B$ iteration with the \GetWeights{} auxiliary function, and $\mathcal{F}_A^0$ and $\mathcal{F}_B^0$ those based on \GetWeightsSub{0}. Of course this time $\mathcal{F}_A$ and $\mathcal{F}_B$ cannot be guaranteed to be continuous because \GetWeights\ might not be. On the other hand, thanks to the continuity of $\GetWeights_0$, it turns out that $\mathcal{F}_A^0$ and $\mathcal{F}_B^0$ are continuous, and by the assumption that \GetWeights performs better or equal than $\GetWeights_0$ we are guaranteed that, for $(\pi^{k_n})_n$ a converging subsequence,
	\begin{equation}
	E(\pi^{k_n+1})
	=
	E(\mathcal{F}_C(\pi^{k_n})) \le E(\mathcal{F}_C^0(\pi^{k_n})) \le E(\pi^{k_n}),
	\end{equation}
	where $C \in \{A,B\}$ stands for the partition corresponding to iteration $k_n$.
	Up to the extraction of a subsequence we can assume that all $C$ iterations are either $A$ or $B$ iterations; let us assume without loss of generality that they are $A$ iterations. Then, taking limits (and using that $\mathcal{F}_A^0$ is continuous and the score is monotonously decreasing) we conclude that
	\begin{equation}
	E(\mathcal{F}_A^0(\hat{\pi})) = E(\hat{\pi}),
	\end{equation}
	which can only hold if $\hat{\pi}$ is already optimal on each cell of partition $A$. The rest of the proof follows analogously to that of Proposition \ref{prop:convergence-sequential-algorithm}, noting that optimality on one partition is enough for showing global optimality (cf.~Remark \ref{rk:only-one-partition-needed}).
\end{proof}

There are several canonical choices to build the function \GetWeights that fulfill the assumptions imposed by either Proposition \ref{prop:convergence-parallel-domdec} or \ref{prop:convergence-parallel-domdec-extended}:

\begin{proposition}
	\label{prop:converging-getweights}
	
	Consider the following choices for $\GetWeights$:
	
	\begin{enumerate}
		\item \GetWeightsSub{safe} $\equiv (1/|\partGeneric|, ..., 1/ |\partGeneric|)$, regardless of the input. 
		\item \GetWeightsSub{swift}: Choose between  $(\theta_J)_{J\in\partGeneric} = (1/|\partGeneric|, ..., 1/ |\partGeneric|)$ and \linebreak $(\theta_J)_{J\in\partGeneric} = (1, ..., 1)$ the one that produces the larger decrement of the score.
		\item \GetWeightsSub{opt}: Optimize over $(\theta_J)_{J\in\partGeneric} \in [0,1]^{|\partGeneric|}$ to produce the best possible decrement.
		\item \GetWeightsSub{line}: Find the scalar $\zeta\in [1/|\partGeneric|,  1]$ that optimizes the decrement by $(\theta_J)_{J\in\partGeneric} = (\zeta, ..., \zeta)$.
	\end{enumerate}
	Then, \GetWeightsSub{safe} satisfies the assumptions of Proposition \ref{prop:convergence-parallel-domdec}, while\linebreak \mbox{\GetWeightsSub{swift}},  \GetWeightsSub{opt} and \GetWeightsSub{line} satisfy those of Proposition \ref{prop:convergence-parallel-domdec-extended} (taking $\GetWeights_{0} =\,$\GetWeightsSub{safe}).
	As a consequence, any of these choices makes Algorithm \ref{alg:UnbalancedDomDecParallel} convergent.
\end{proposition}

\begin{proof} Since \GetWeightsSub{safe} is trivially continuous, and the other functions by construction perform at least as well as \GetWeightsSub{safe}, all that remains to be shown is that \GetWeightsSub{safe} induces a decrement (item \ref{item:parallel-decrement} of the assumptions in Proposition \ref{prop:convergence-parallel-domdec}).
	
	Let $\pi$ be a given plan, $\partGeneric$, be a partition and $\tilde{\pi}_J = S_J(\pi)$ for each $J$ in $\partGeneric$. The updated plan according to \GetWeightsSub{safe} is given by
	\begin{equation*}
	\bar{\pi}
	=
	\sum_{J \in \mathcal{J}}
	\left(
	1 - \frac{1}{|\partGeneric|}
	\right)\pi_J
	+
	\frac{1}{|\partGeneric|}
	\tilde{\pi}_J.
	\end{equation*}
	Let us show that unless all cell plans are locally optimal, $\bar{\pi}$ has a strictly lower score than  $\pi$. First we perform the trivial manipulation (below $E$ is the objective in \eqref{eq:unbalanced-entropic-ot}):
	\begin{align*}
	E(\bar{\pi})
	&=
	E\left(
	\sum_{J \in \mathcal{J}}
	\left(
	1 - \frac{1}{|\partGeneric|}
	\right)\pi_J
	+
	\frac{1}{|\partGeneric|}
	\tilde{\pi}_J
	\right)
	=
	E\left(
	\frac{1}{|\partGeneric|}
	\sum_{J \in \mathcal{J}}
	\pi-\pi_J+\tilde{\pi}_J
	\right)
	\\
	\intertext{Now by Jensen's inequality, and then the optimality of $\tilde{\pi}_J$ on its subproblem:}
	&\le 
	\sum_{J \in \mathcal{J}}
	\frac{1}{|\partGeneric|}
	E(\pi-\pi_J+\tilde{\pi}_J)
	\le 
	\sum_{J \in \mathcal{J}}
	\frac{1}{|\partGeneric|}
	E(\pi)
	= E(\pi).
	\end{align*}
	Crucially, by strict convexity of each cell subproblem, the last inequality can only turn into an equality if $\pi_J$ is already optimal for each subproblem. Thus, the update given by \GetWeightsSub{safe} fulfills the assumptions of Proposition \ref{prop:convergence-parallel-domdec}. 
	
\end{proof}

\section{Numerical experiments}
\label{sec:numerics}
All the experiments in this Section may be reproduced with the code available online\footnote{\scriptsize \url{https://github.com/OTGroupGoe/DomainDecomposition/tree/main/unbalanced-domdec-paper}}.

\subsection{Implementation details}
\label{sec:implementation}
\paragraph{Choice of cost and divergence}
In our experiments we let $c(x,y)=|x-y|^2$, which is by far the most relevant choice. Adaptation to other costs is straight-forward. Moreover, we choose $D_1 = D_2 = \lambda \KL$, which is one of the most common divergences for unbalanced transport \cite{Frogner2015,ChizatUnbalanced2016,Liero2018} and satisfies Assumption \ref{ass:divergence}.
(We expect that the steps described below can easily be adapted to other appropriate entropy functions.)
For primal and dual candidates $\pi_J, (\alpha_J, \beta_J)$ the optimality conditions \eqref{eq:diagonal-scaling} become
\begin{equation}
\label{eq:optimality-conditions-KL}
\exp\left(-\frac{\alpha_J}{\lambda}\right) = \RadNikD{\proj_X \pi_J}{\mu_J}
\quad
\tn{and}
\quad 
\exp\left(-\frac{\beta_J}{\lambda}\right)
= 
\RadNikD{\proj_Y \pi_J}{\nu}
+ \RadNikD{\nu_{-J}}{\nu}.
\end{equation}

It can be shown that for this problem $\sqrt{\lambda}$ acts as characteristic length-scale: for distances much smaller than $\sqrt{\lambda}$ the problem behaves similar to a balanced transport problem, for much larger distances virtually no transport occurs and it is cheaper to accept the marginal penalty instead. In the following experiments we will therefore mostly give the value of $\sqrt{\lambda}$ for better interpretability.

\paragraph{Solving cell subproblems} We solve the dual cell subproblems with the Sinkhorn algorithm for unbalanced transport \cite{ChizatUnbalanced2016}. For a given initialization $\beta^0$, the Sinkhorn iterations are given as the solutions to the following equations for $\ell \geq 0$:
\begin{equation}
\label{eq:sinkhorn-X-iteration}
\exp\left(-\tfrac{(\veps + \lambda)\alpha^{\ell+1}(x)}{\veps\lambda} \right)
=
\int_Y
\exp\left(\tfrac{\beta^\ell(y) - c(x,y)}{\veps} \right)\diff \nu(y)
\end{equation}
and
\begin{equation}
\label{eq:sinkhorn-Y-iteration}
\exp\left(-\tfrac{\beta^{\ell+1}(y)}{\veps} \right)
\left[
\exp\left(-\tfrac{\beta^{\ell+1}(y)}{\lambda}\right)
- \RadNik{\nu_{-J}(y)}{\nu}
\right]
=
\int_{X_J}\hspace{-1.5mm}
\exp\left(\tfrac{\alpha^{\ell+1}(x) - c(x,y)}{\veps} \right)\diff \mu_J(x).
\end{equation}
With these we associate the primal iterates for $\ell \geq 1$:
\begin{equation}
\label{eq:sinkhorn-primal}
\pi^{\ell} = \exp\left(\frac{\alpha^{\ell} \oplus \beta^{\ell}-c}{\veps}\right) \cdot \mu \otimes \nu.
\end{equation}
The $X$-half iteration \eqref{eq:sinkhorn-X-iteration} is straightforward to solve and yields the classical Sinkhorn iteration for $\KL$-regularized unbalanced transport \cite[Table 1]{ChizatUnbalanced2016}. 
Unfortunately, due to the $Y$-submarginal term involving $\nu_{-J}$, the $Y$-half iteration \eqref{eq:sinkhorn-Y-iteration} cannot be solved in closed form.
Nevertheless, since for each $y$ \eqref{eq:sinkhorn-Y-iteration} corresponds to the minimization of a one-dimensional convex function, it can be solved efficiently with the Newton algorithm.  
To avoid numerical over- and underflow it is convenient to reformulate the problem in the log-domain. Introducing $z^{\ell+1}(y) := \log \int_{X_J} \exp[(\alpha^{\ell+1} - c(\cdot, y))/\veps] \diff \mu_J$, we can rewrite \eqref{eq:sinkhorn-Y-iteration} as
\begin{equation}
\label{eq:proxdiv-problem}
\log\left[
\RadNikD{\nu_{-J}}{\nu}(y)
+
\exp
\left(
\frac{\beta^{\ell+1}(y)}{\veps}
+z^{\ell+1}(y)\right)
\right]
+\frac{\beta^{\ell+1}(y)}{\lambda}
= 0,
\end{equation} 
which can be computed in a log-stabilized way.

\paragraph[Computing and handling nu\_\{-J\}]{Computing and handling $\nu_{-J}$} For a given plan $\pi$, the $Y$-submarginal at cell $J$ is given by $\nu_{-J} = \proj_Y (\pi - \pi_J)$. This can be computed efficiently if our implementation stores the basic cell $Y$-marginals $(\nu_i)_{i\in I}$ (cf. Section \ref{sec:domdec}), since
\begin{equation}
\nonumber
\nu_J \assign \sum_{i\in J} \nu_i 
\quad \tn{for each $J\in \partGeneric$,}
\quad
\proj_Y\pi = \sum_{J\in \partGeneric} \nu_J,
\quad
\tn{and therefore}
\quad
\nu_{-J}
=
\proj_Y \pi - \nu_J.
\end{equation}
Therefore it is sufficient to compute the global marginal $\proj_Y \pi$ once per iteration.
In practice we represent the marginals $(\nu_J)_{J\in \partGeneric}$ as measures with sparse support by truncating very small values as described in \cite{BoSch2020}. This implies that $\nu_{-J}$ equals $\proj_Y\pi$ outside of a sparse set.
When we are close to optimality (as for example in a coarse-to-fine algorithm), then we expect that changes to $\nu_J$ are confined to the support of $\nu_J$. (Other changes to $\nu_J$ can still occur by mass traveling between cells when alternating between the composite partitions.) In this case it then suffices to only consider the restriction of $\nu_{-J}$ to the sparse support of $\nu_J$ in the cell problem.
The validity of this strategy is confirmed by the numerical experiments in Section \ref{sec:multiscale}.

\paragraph{Stopping criterion for the Sinkhorn algorithm on the cell problems}
The Sinkhorn algorithm will usually not terminate with an exact solution after a finite number of iterations. An approximate stopping criterion has to be chosen. In the case of unbalanced transport the primal dual gap between $\pi^\ell$, \eqref{eq:sinkhorn-primal} and $(\alpha^\ell,\beta^\ell)$ can be used (this does not work in the balanced case, since the primal iterate will typically not be a feasible coupling).

Using the identities \eqref{eq:sinkhorn-X-iteration} to \eqref{eq:sinkhorn-primal} one finds (after a lengthy computation):
\begin{align}
\label{eq:PD-gap}
\PD(\pi_J^\ell, (\alpha_J^\ell, \beta_J^\ell) \mid \nu_{-J})
&\assign
E(\pi_J^\ell \mid \nu_{-J}) - D(\alpha_J^\ell, \beta_J^\ell \mid \nu_{-J})
\\
&=\lambda \KL(\proj_X \pi_J^\ell \mid e^{-\alpha_J^\ell/\lambda} \mu_J).
\label{eq:PD-gap-KL}
\end{align}
We will use $\PD(\pi_J^\ell, (\alpha_J^\ell, \beta_J^\ell) \mid \nu_{-J})/\lambda \leq \delta$ for some tolerance threshold $\delta$ as stopping criterion, which is appealing for several reasons.
First, by \eqref{eq:PD-gap-KL}, it measures how far $\proj_X\pi_J^k$ and $\alpha_J^k$ are from verifying their primal-dual optimality condition \eqref{eq:optimality-conditions-KL}. Second, it does not blow up in the limit $\lambda \rightarrow \infty$, but tends to $\KL(\proj_X \pi_J^\ell \mid \mu_J)$, which quantifies the $X$-marginal constraint violation. Finally, by letting $\nu_{-J} = 0$ and $J = I$ (i.e.~there is only a single partition cell), \eqref{eq:PD-gap-KL} also applies to the global problem.

We found that the choice of the tolerance $\delta$ is a delicate question, especially for intermediate values of $\lambda$.
Figure \ref{fig:convergence-plateau} shows the evolution of the marginal error for a global (i.e.~no domain decomposition is involved in this plot) unbalanced problem with $\lambda \KL$ soft marginal.
We observe two stages of convergence. First, the error \eqref{eq:PD-gap-KL} decreases at a linear rate, which roughly matches that of the balanced problem ($\lambda = \infty$). However, once a certain threshold (which decreases as $\lambda$ increases) is reached, the convergence rate deteriorates noticeably. The effect is more pronounced for larger $\lambda$. As a result, the number of iterations needed for achieving a given tolerance (shown in Figure \ref{fig:convergence-plateau}, right) depends non-trivially on $\lambda$.
\begin{figure}[bt]
	\includegraphics[width=0.99\linewidth]{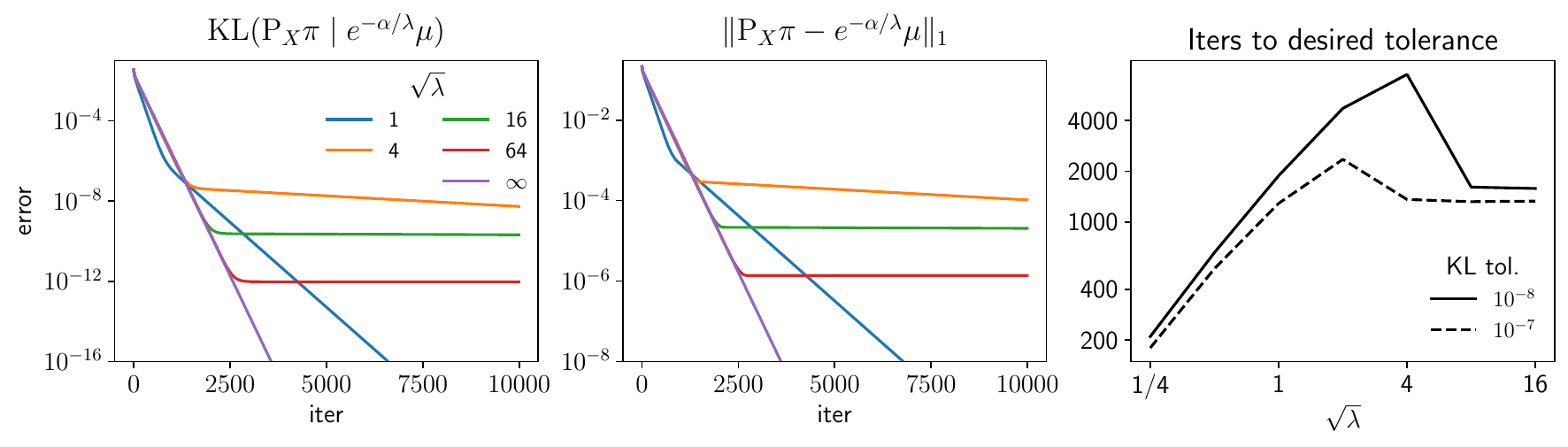}
	\caption{The left and center panels show the evolution of two notions of the marginal error for an unbalanced transport problem between two images of size $256 \times 256$ and regularization strength $\veps = (8\Delta x)^2$ (where $\Delta x=1/256$ is the image grid resolution). The $\KL$ criterion on the left is the primal-dual gap, normalized by $\lambda$ (see \eqref{eq:PD-gap-KL}), while the $\LL^1$ criterion in the center is a heuristic adaptation of the $\LL^1$ error that is often used for balanced transport.
	For finite $\lambda$, after an initial stage of fast convergence the rate decreases sharply.
	On the right, we show the impact that the plateaus in the left and center plots have on the number of iterations needed to reach a given tolerance for different ranges of $\lambda$. 
	}
	\label{fig:convergence-plateau}
\end{figure}

We emphasize that his behavior concerns both the global unbalanced problem (corresponding to $\nu_{-J} = 0$) as well as the domain decomposition cell problems. It was already observed in \cite[Figure 5]{ChizatUnbalanced2016} for a fixed value of $\lambda$. Therefore, understanding the cause of this behavior remains an interesting topic for future investigations.
In the meantime, as a practical remedy we recommend to numerically explore the position of the plateau in a given set of problems and then choose $\lambda$ in a compromise of acceptable precision and fast convergence.

\paragraph{Balancing} 
In balanced domain decomposition \cite{BoSch2020}, the Sinkhorn algorithm for the cell problems is terminated after a $Y$-half iteration, which guarantees that $\nu_J^k = \nu_J^{k-1}$ and thus the global $Y$-marginal $\proj_Y \pi^k = \sum_J\nu_J^k$ is preserved during the execution of the algorithm.
But some marginal error occurs on the $X$-marginal, which means that in general $ |\nu_i^k| \neq |\mu_i| $ and these errors will accumulate over time. In \cite{BoSch2020} a \emph{balancing} step is introduced to correct this marginal error by transferring a small amount of mass between basic cells within a composite cell, in such a way that $\nu_J^k$ is preserved but also $|\nu_i^k| = |\mu_i|$ for each $i\in J$ at the end of every domain decomposition iteration.

In unbalanced transport, since the marginal constraint is replaced by a more flexible marginal penalty, one might hope that such balancing procedure were unnecessary. However, as $\lambda\to\infty$ the problem will behave increasingly similar to the balanced limit, which means that it becomes increasingly sensitive to small mass fluctuations in the basic cells.
Indeed we found that in practice an adapted balancing step also improves convergence of unbalanced domain decomposition.

We chose the following adaptation.
For cell $J \in \partGeneric$, let $(\alpha_J, \beta_J)$ be the final Sinkhorn iterates and let $(\nu_i^k)_{i\in J}$ be the basic cell marginals. Then perform an additional Sinkhorn half-iteration to get the $X$ marginal of the subsequent plan, that we denote by $\bar{\mu}_J$. Next, for each $i\in J$ introduce
\begin{equation}
\label{eq:balancing}
m_i \assign 
\bar{\mu}_J(X_i)
\frac{\nu_J^k(Y)}{\bar{\mu}_J(X_J)}.
\end{equation}
The balancing step now shifts mass horizontally (i.e., without changing its $y$-coordinate) between the basic cells in $J$ such that $|\nu_i^k| = m_i$ for all $i\in J$. This is possible since $\sum_{i\in J} m_i = |\nu_J^k| = \sum_{i\in J} |\nu_i^k|$. 
For balanced transport this reduces to the scheme of \cite{BoSch2020}, since in this case $\bar{\mu}=\mu$ and therefore $m_i=\mu(X_i)$.
The intuition for unbalanced transport is that $\bar{\mu}$, obtained after an $X$-half iteration, is a better estimate of the (unknown) optimal $X$-marginal than $\proj_X \pi^k$, which was obtained after a $Y$-half iteration. Hence, adjusting the basic cell masses to agree with $\bar{\mu}$ yields faster convergence. This intuition is confirmed by numerical experiments, especially for large $\lambda$.

\subsection{Different parallelization weight functions}
\label{sec:single-scale}
In this section we compare different ways and weight functions to parallelize the partition for loop in Algorithm \ref{alg:UnbalancedDomDecParallel}. All mentioned versions of the $\GetWeights$ function were defined in \eqref{eq:GWFast} and Proposition \ref{prop:converging-getweights}. The methods and their shorthands are as follows:
	
	\texttt{sequential}: The sequential Algorithm \ref{alg:UnbalancedDomDecSequential}.
	
	\texttt{safe}: The parallel algorithm with \GetWeightsSub{safe}. This is straightforward to implement but the stepsize becomes impractically small as the partition size increases.
	
	\texttt{greedy}: The parallel algorithm with the naive greedy choice \GetWeightsSub{greedy} that always returns $(1, ..., 1)$, \eqref{eq:GWFast}. This is \emph{not} guaranteed to converge according to our theory and it will also frequently not converge in practice.
	
	\texttt{swift}: The parallel algorithm with \GetWeightsSub{swift}, which allows a greedy update whenever it is beneficial, and falls back to \GetWeightsSub{safe} otherwise.
	
	\texttt{opt}: The parallel algorithm with \GetWeightsSub{opt}. This is guaranteed to yield the largest possible instantaneous decrement, but it can become computationally expensive to solve the weight optimization problem.
	
	The main hurdle consists in computing and evaluating the score of the global plan. Note that in the rest of the algorithm the global plan can be stored in a very compact sparse structure, as briefly explained in Section \ref{sec:implementation} (see \cite{BoSch2020} and \cite{hybrid} for additional details). 
	The sparsity of the plan is also undermined by the \texttt{opt} strategy, since the cell update is computed as a convex combination of former and new cell plans, and thus typically features more non-negligible entries than either the former or the newly found optimal plan. We will expand on this when we compare the results.

	\texttt{linesearch}: The parallel algorithm with \GetWeightsSub{line}. 
	Even though we expect this approach to be less computationally expensive than the \texttt{opt} strategy, one still needs to repeatedly instantiate the transport plan, which adds additional overhead specially in large problems.
	Besides, since the convex combination weights are the same for each cell, one can expect the sparsity produced by this strategy to be worse than that of \texttt{opt}.

	\texttt{staggered}: Subdivide each partition into $2^d$ batches of staggered cells, in such a way that neither partition contains adjacent composite cells (see Figure \ref{fig:subpartitions}). Then apply the \texttt{swift} strategy to each of the subpartitions. Intuitively, for small values of the regularization $\veps$ and close to convergence, very little overlap between the $Y$-marginals $\nu_J$ in each of the subpartitions is expected, and thus the cell problems become approximately independent, i.e.~the probability for triggering the \GetWeightsSub{safe} fallback is small and the greedy update will be applied more often.
	
According to Propositions \ref{prop:convergence-parallel-domdec}, \ref{prop:convergence-parallel-domdec-extended}, and \ref{prop:converging-getweights} all strategies (except for \texttt{greedy}) are guaranteed to converge.

In Figures \ref{fig:strategies-lebesgue-product-joint} and \ref{fig:strategies-lebesgue-product-scores} we compare the different proposed strategies on a toy problem with $\mu = \nu$ being a homogenous discrete measure on the equispaced grid $X = Y=\{0, \tfrac{1}{N}, ,...,1-\tfrac{1}{N}\}$, for $N = 32$. 
The composite partition $\partA$ is obtained by aggregating neighboring pairs of points into composite cells, $\partB$ is the staggered partition with two singletons $\{0\}$ and $\{1-1/N\}$ at the beginning and the end.
We set $\lambda = 1$, and $\veps = 2/N^2$ (i.e.~the blur scale $\sqrt{\veps}$ is on the order of the grid resolution).
The optimal plan will be approximately concentrated along the diagonal with a blur width of a few pixels.
As initial iterate we choose the product measure $\mu \otimes \nu$.

Figure \ref{fig:strategies-lebesgue-product-joint} shows both the short and long term behavior of the domain decomposition iterations, while \ref{fig:strategies-lebesgue-product-scores} compares the sub-optimality gap of the objective over the number of iterations, where a reference solution was computed with a global Sinkhorn algorithm with very small tolerance. 
We do not compare the runtime of the different strategies, since their performance on small problems cannot be extrapolated to the large-scale problems in Section \ref{sec:multiscale}.
In the following we briefly discuss the behavior of each strategy. 

\begin{figure}[btp]
	\centering
	\begin{equation*}
	\hspace{13mm}
	\xleftrightarrow{\hspace{13mm}\text{Short time behavior}\hspace{13mm}}
	\hspace{2mm}
	\xleftrightarrow{\hspace{13mm}\text{Long time behavior}\hspace{13mm}}
	\hspace{10mm}
	\end{equation*}
	\vskip -3mm
	\includegraphics[width=0.99\linewidth]{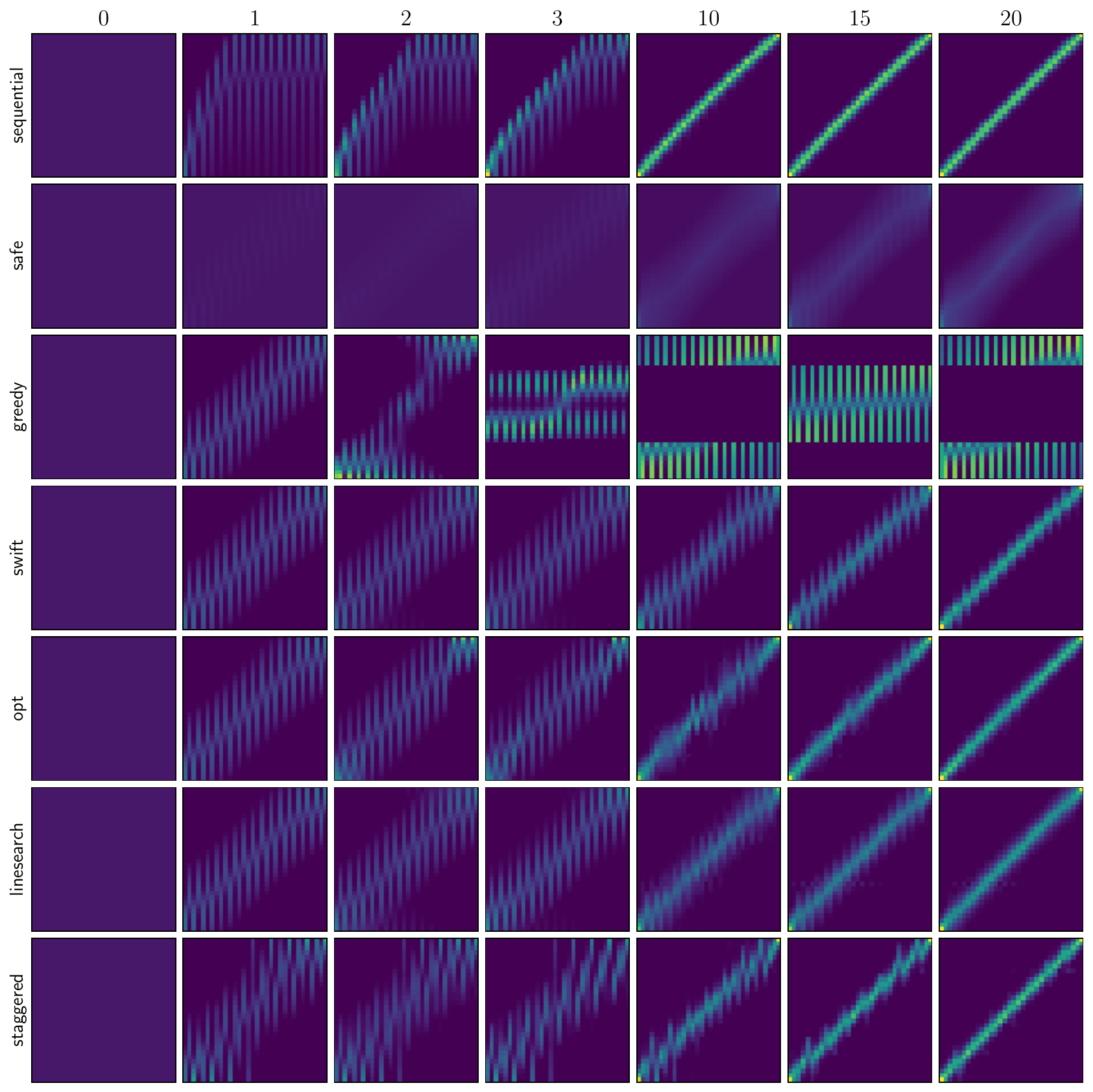}
	\caption{Behavior of the different parallelization approaches for $N = 32$. The column titles denote the iteration number $k$.
	}
	\label{fig:strategies-lebesgue-product-joint}
\end{figure}
\begin{figure}
	\centering
	\includegraphics[width=0.99\linewidth]{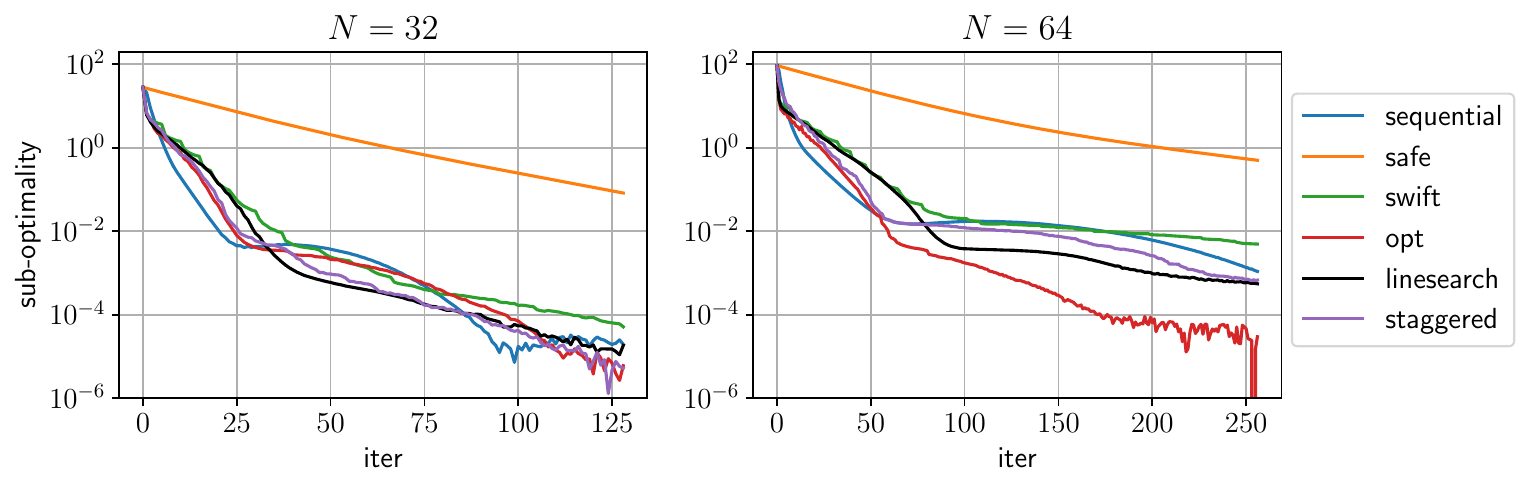}
	\caption{Evolution of the sub-optimality gap for the different parallelization approaches.
	}
	\label{fig:strategies-lebesgue-product-scores}
\end{figure}

\paragraph{\texttt{sequential}} In Figure \ref{fig:strategies-lebesgue-product-joint} we observe a rapid formation of the diagonal structure, much faster than what we would expect in the balanced case (where we need approximately $N$ iterations, see  \cite{asymptotic_domdec_springer}). 
Also unlike in the balanced case, the evolution is not symmetrical, which comes from the order in which the cell problems are solved and their interdependence.
In terms of iterations, the global score is decreased in a way that is comparable with the best parallel approaches (see Figure \ref{fig:strategies-lebesgue-product-scores}), which makes sense since each cell problem can already react to the changes by the previous cells. Of course, due to the sequential nature, a single iteration will require much more time. We attribute the slight increment of the score around iteration $k = 30$ to the inexact solution of the Sinkhorn subproblem.

\paragraph{\texttt{safe}} The small step size imposed by the \texttt{safe} approach leads to a very slow evolution of the iterates (see both figures). On larger problems this will become even worse, since the step size scales like $1/|\partGeneric|$. This renders the \texttt{safe} approach impractical for an efficient, large-scale solver. 

\paragraph{\texttt{greedy}} The \texttt{greedy} scheme leads to non-convergent behavior due to a lack of coordination between the cells. In the first domain decomposition iteration the cells create a surplus of mass in the central region of $Y$. The next iteration compensates for this imbalance, but since each subproblem acts on its own, they overshoot and create a new accumulation of mass at the extremes of $Y$. 
This oscillation pattern somewhat stabilizes into a cycle. Hence, without an additional safeguard, the \texttt{greedy} approach fails to provide a convergent algorithm.

\paragraph{\texttt{swift}} The \texttt{swift} approach is a simple way to add a safeguard to the \texttt{greedy} approach. By switching between the \texttt{greedy} and \texttt{safe} strategies based on their decrements, the algorithm becomes convergent.
However, the fallback to \texttt{safe} is rather inefficient, as can be inferred from Figure \ref{fig:strategies-lebesgue-product-scores}. Especially during the early iterations, a sharp decrement of the score (corresponding to a \texttt{greedy} iteration) is followed by several iterations with slow decrement (corresponding to \texttt{safe} iterations).

\paragraph{\texttt{opt}} As shown in Figures \ref{fig:strategies-lebesgue-product-joint} and \ref{fig:strategies-lebesgue-product-scores}, the \texttt{opt} strategy is among the fastest in terms of iteration numbers. Its main drawback is that computing its convex combination weights hinges on the solution of the auxiliary optimization problem
\begin{equation}
\min_{(\theta_J)_{J\in \partGeneric}} 
E\left(\sum_{J\in \partGeneric} (1- \theta_J) \pi_J + \theta_J\tilde{\pi}_{J} \right),
\qquad 
\tn{s.t. } 0 \le \theta_J \le 1
\tn{ for all }J\in \partGeneric.
\label{eq:opt-weights-problem}
\end{equation}
Although \eqref{eq:opt-weights-problem} is a convex optimization problem, solving it in every iteration adds considerable overhead due to the need to instantiate the global plan for different choices of $(\theta_J)_J$.

Furthermore, as shown in Figure \ref{fig:strategies-lebesgue-product-sparsity}, the \texttt{opt} strategy features a slower reduction in the number of non-negligible plan entries compared to the \texttt{staggered} strategy.

\begin{figure}[h]
	\centering
	\includegraphics[width=\linewidth]{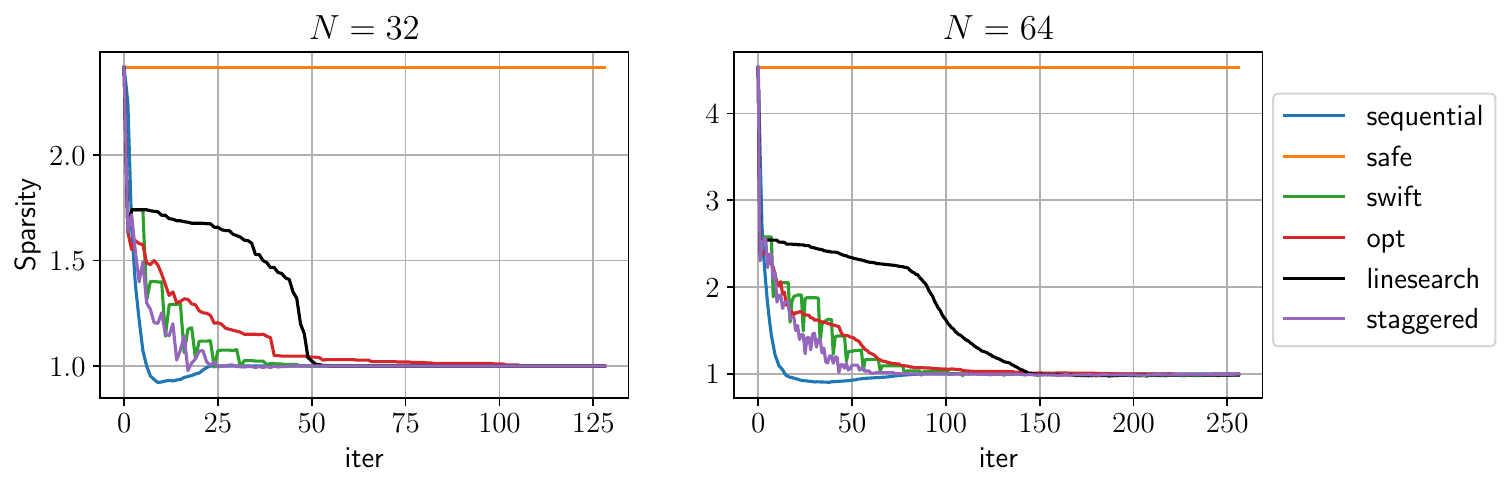}
	\caption{Evolution of the coupling sparsity, which we define as the number of plan entries above the truncation threshold $\textsc{Thr} = 10^{-14}$, normalized to the number of entries (above the truncation threshold) of the optimal plan.
	}
	\label{fig:strategies-lebesgue-product-sparsity}
\end{figure}

\paragraph{\texttt{linesearch}}
The \texttt{linesearch} strategy yields one of the best performances in Figure \ref{fig:strategies-lebesgue-product-scores}, close to the \texttt{opt} strategy and slightly better than \texttt{staggered} in this problem regime.
However, this approach again involves the repeated instantiation of the transport plan, which for large problem sizes adds considerable overhead. Furthermore, as shown in Figure \ref{fig:strategies-lebesgue-product-sparsity}, the number of non-negligible entries decays very slowly, as a consequence of \GetWeightsSub{line} assigning the same convex combination weights to all cell problems. 
This will result in a degradation of performance as the problem size increases.

\paragraph{\texttt{staggered}} 
As exemplified by Figures \ref{fig:strategies-lebesgue-product-joint} and \ref{fig:strategies-lebesgue-product-scores}, the \texttt{staggered} strategy achieves fast, consistent decrements without relying on auxiliary optimization problems nor excessive calls to the \texttt{safe} strategy.
This is achieved by the division of each partition into staggered batches, which in practice greatly reduce the overlap between cell $Y$-marginals (after some initial iterations). 

The \texttt{staggered} strategy is also among the fastest in terms of reducing the number of non-negligible plan entries (cf. Figure \ref{fig:strategies-lebesgue-product-sparsity}), only outperformed by the \texttt{sequential} strategy.

Although the \texttt{opt} and \texttt{linesearch} strategies may provide faster decrements in terms of iteration numbers, the \texttt{staggered} strategy avoids expensive plan instantiation and shows an excellent performance in providing a sparse plan, so we consider it to be preferable, in particular on larger problems. 

\subsection{Using only single partition}
\label{sec:one-partition}

As pointed out in Remark \ref{rk:only-one-partition-needed}, domain decomposition for unbalanced optimal transport only requires one partition for convergence, in contrast with the balanced case that requires two partitions fulfilling a certain joint connectivity property \cite{BoSch2020}.
Due to this requirement in the balanced problem, we expect that using two partitions will prove advantageous also in unbalanced problems.
The intuitive reason is that employing two, interleaved partitions allows mass to `flow' from one cell to another as the iterations progress. Conversely, when restricted to a single partition the only mechanism available for mass exchange involves destroying mass in one cell and creating it in another. However, since the cell iterations are independent, different cells cannot coordinate this procedure, and we expect the ensuing convergence to be slow compared to the case of interleaved partitions.

We test this intuition in Figure \ref{fig:strategies-lebesgue-product-joint-npart-1}, where we show the behavior of the domain decomposition algorithm using a single partition for a selection of the parallelization strategies. The respective evolution of the suboptimality gap is shown in Figure \ref{fig:strategies-lebesgue-product-scores-npart-1}.
As expected, using a single partition slows down convergence noticeably. 
This effect becomes stronger as the iterations progress, since close to convergence incentives for mass creation or destruction become smaller. 
We conclude that employing two interleaved partitions has beneficial effects on the convergence of the domain decomposition algorithm even in the unbalanced case, and hence we employ this partition structure in Section \ref{sec:multiscale}.

\begin{figure}[btp]
	\centering
	\includegraphics[width=0.99\linewidth]{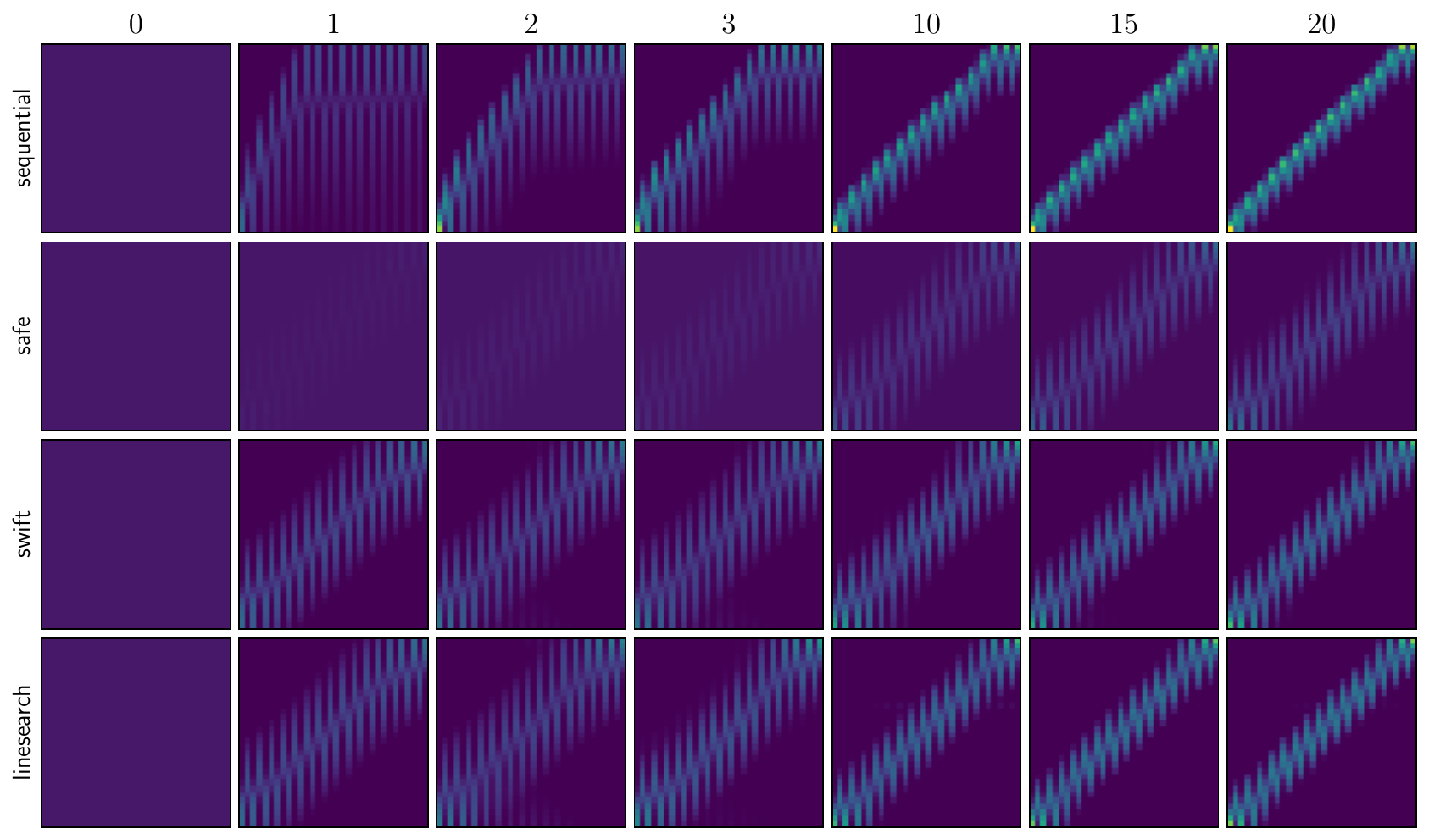}
	\caption{Behavior of the different parallelization approaches for $N = 32$, using a single partition. The column titles denote the iteration number $k$.
	}
	\label{fig:strategies-lebesgue-product-joint-npart-1}
\end{figure}

\begin{figure}
	\centering
	\includegraphics[width=0.99\linewidth]{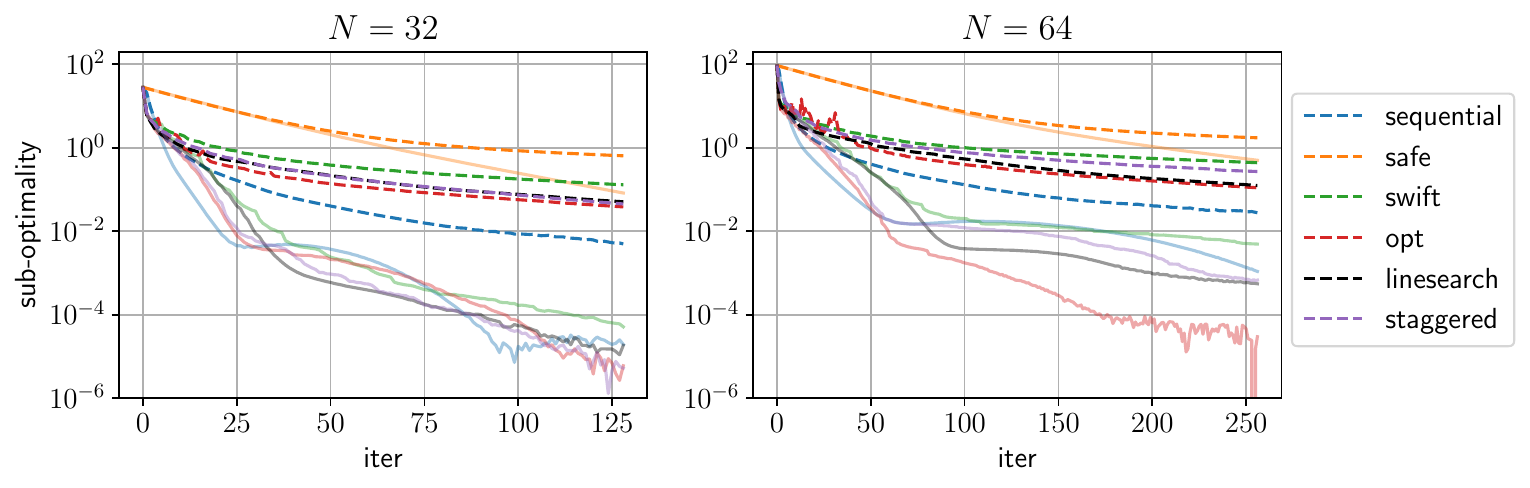}
	\caption{Evolution of the sub-optimality gap for the different parallelization approaches, using a single partition. 
	The corresponding results using two partitions (see Figure \ref{fig:strategies-lebesgue-product-scores}) are shown in faded colors and continuous lines.
	}
	\label{fig:strategies-lebesgue-product-scores-npart-1}
\end{figure}

\subsection{Large scale examples}
\label{sec:multiscale}

\subsubsection{Setup}
\label{sec:numerics-setup}

\paragraph{Compared algorithms, implementations, and hardware} In this Section we compare unbalanced domain decomposition to the global unbalanced Sinkhorn algorithm.
Our implementation for unbalanced domain decomposition is based on the GPU implementation of balanced domain decomposition described in \cite[Appendix A]{hybrid}\footnote{Code available at \url{https://github.com/OTGroupGoe/DomainDecomposition}.}. 
Its main features are sparse storage of cell marginals $(\nu_i)_{i\in I}$ in a geometric `bounding box structure', and a Sinkhorn cell solver that is tailored specifically for handling a large number of small problems in parallel.
For the global Sinkhorn solver we also use the implementation discussed in \cite{hybrid}, which compares favorably to \texttt{keops/geomloss} \cite{FeydyDissertation} in the range of problem sizes considered (see \cite[Figure 10]{hybrid}).
In the following we will refer to these two algorithms as \DomDecGPU\ and \SinkhornGPU, respectively.

The experiments were run on an Intel Xeon Gold 6252 CPU with 24 cores (we use 1) and an NVIDIA V100 GPU with 32 GB of memory.

\paragraph{Stopping criterion and measure truncation} 
We use \eqref{eq:PD-gap-KL} divided by $\lambda$ (see discussion above) to measure the suboptimality error of the Sinkhorn iterations.
As error threshold we use $||\mu_\partGeneric|| \cdot \tn{Err}$, where for \DomDecGPU\ $\mu_\partGeneric$ is the collection of the $X$-marginals of the cell problems the cell Sinkhorn solver is handling simultaneously, and for \SinkhornGPU\ is simply $\mu$. We use $\tn{Err} = 2\cdot 10^{-5}$, which yields a good balance between accuracy and speed. 

In \DomDecGPU, the basic cell marginals $(\nu_i)_{i\in I}$ are truncated at $10^{-15}$ and stored in bounding box structures (see \cite[Appendix]{hybrid} for more details).

\paragraph{Test data} We use the same problem data as in \cite{BoSch2020} and \cite[Appendix]{hybrid}: images with dimensions $N \times N$, with $N=2^\ell$ and $\ell$ ranging from $\ell=6$ to $\ell = 11$, i.e.~images between the sizes $64 \times 64$ to $2048 \times 2048$. The images are Gaussian mixtures restricted to $X = [0,1]^2$ with random variances, means, and magnitudes. For each problem size we generated 10 test images, i.e.~45 pairwise transport problems, and average the results. 
An example of the problem data is shown in Figure \ref{fig:ExampleImages}. 

\begin{figure}[hbt]
	\centering
	{\def\imgw{2.4cm}%
		\begin{tikzpicture}[x=\imgw,y=\imgw,img/.style={inner sep=0pt,draw=black,line width=1pt,anchor=north west}]
		\begin{scope}[shift={(0,-1.1)}]
		\foreach \x/\y/\l in {0/0/3,1/0/4,2/0/5,3/0/6,4/0/7,5/0/8} {
			\node[img] at (1.1*\x,-1.1*\y)[label=below:{\small $\ell=\l$}]{\includegraphics[width=\imgw]{muY_l\l.png}};
		}
		\end{scope}
		\end{tikzpicture}%
	}%
	\caption{A marginal with size $256 \times 256$ shown at different subsampling layers.}
	\label{fig:ExampleImages}
\end{figure}

\paragraph[Multi-scale and eps-scaling]{Multi-scale and $\veps$-scaling} All algorithms implement the same strategy for multi-scale and $\veps$-scaling as in \cite[Section 6.4]{BoSch2020}. At every multiscale layer, the initial regularization strength is $\varepsilon = 2(\Delta x)^2$, where $\Delta x$ stands for the pixel size, and the final value is $\varepsilon = (\Delta x)^2 / 2$. In this way, the final regularization strength on a given multiscale layer coincides with the initial one in the next layer. On the finest layer $\varepsilon$ is decreased further to the value $(\Delta x)^2 / 4$ implying a relatively small residual entropic blur.

\paragraph{Double precision}
All algorithms use double floating-point precision, since single precision causes a degradation in the accuracy for problems with $N \gtrsim 256$ and small $\veps$. For more details see \cite[Appendix]{hybrid}, where this behavior was originally reported.

\begin{figure}
	\centering
	\includegraphics[width=0.99\linewidth]{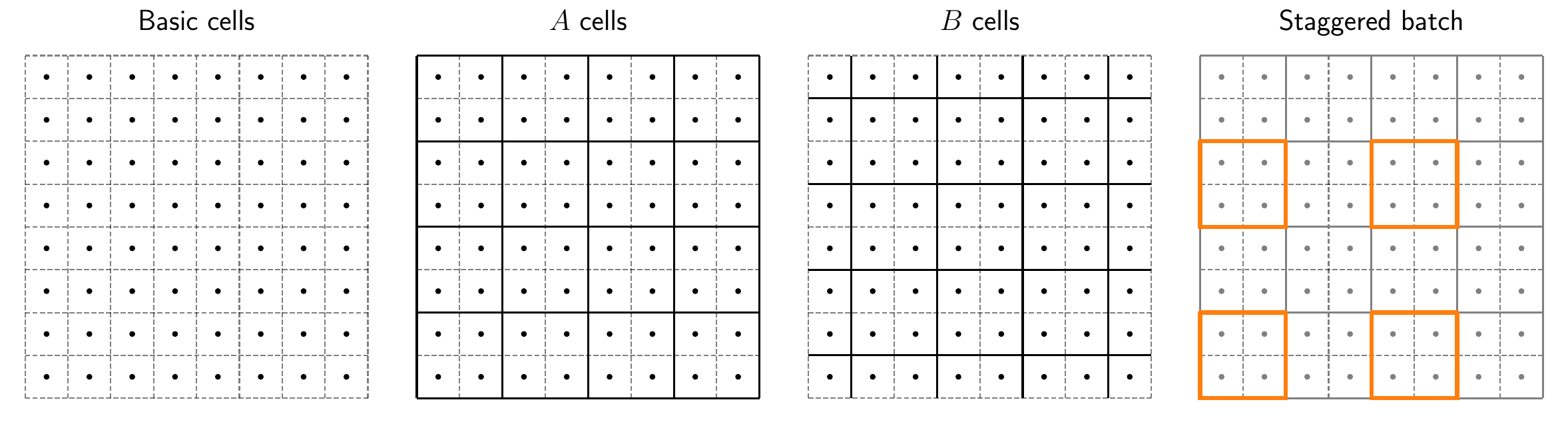}
	\caption{Partition structure and staggered batches.}
	\label{fig:subpartitions}
\end{figure}

\paragraph{Basic and composite partitions} 
We generate basic and composite partitions as in \cite{BoSch2020}: for the basic partition we divide each image into blocks of $s\times s$ pixels (where $s$ is a divisor of the image size), while composite cells are obtained by grouping cells of the same size, as shown in Figure \ref{fig:subpartitions}.
For the $B$ cells we pad the image boundary with a region of zero density, such that all composite cells have the same size, to simplify batching.
As in \cite{BoSch2020} and \cite{hybrid}, we observe experimentally that the choice of cell size $s=4$ yields the best results.

\paragraph{\texttt{staggered} strategy and batching}
Based on Section \ref{sec:single-scale} we use the \texttt{staggered} strategy for parallelizing the unbalanced cell problems. 
Parallelizing the cell problems in batches of four staggered grids (see Figure \ref{fig:subpartitions}) reduces the overlap between their $Y$-marginal supports and thus increases the chances of accepting the greedy iteration.
In practice we are somewhat lenient when accepting the greedy solution, since the finite error in the cell Sinkhorn stopping criterion can sometimes lead to small increases in the objective. We allow for an increase of up to 0.5\% in the primal score without triggering the safeguard. For large values of $\lambda$ it might be practical to further increase this threshold since small deviations in the marginal lead to strong increases in the score (see also the discussion on the balancing step in Section \ref{sec:implementation}).

\paragraph{Sinkhorn warm start} 
A drawback of the \texttt{staggered} strategy is that one is forced to use at least $2^d$ batches, which causes considerable overhead on small problems. 
For this reason, even in \DomDecGPU\ we solve the first multiscale layers with the global Sinkhorn solver, and only switch to domain decomposition starting at layer  $\ell = 8$ (for $N \ge 512$), or $\ell = \floor{\log_2(N)}-1$ (for $N < 512$) (i.e., for small problems only the last layer is solved with domain decomposition). This combines the efficiency of GPU Sinkhorn solvers in small-to-medium sized problems with the better time complexity of domain decomposition for larger problems.
The tolerance of the lower resolution Sinkhorn solver is set to one forth of the domain decomposition tolerance, with the objective of minimizing the mismatch between $\proj_X \pi$ and the optimal first marginal, since as explained in \cite{BoSch2020} and Section \ref{sec:implementation} domain decomposition suffers from balancing problems when the first marginal is not kept close to optimal.

\subsubsection{Results}

\paragraph[Influence of lambda]{Influence of $\lambda$}
Figure \ref{fig:deformation-map} illustrates the influence of $\lambda$ on the optimal transport map for an example problem. For very small $\lambda$ mass is barely transported since the marginal discrepancy in the $\KL$-terms is very cheap. As $\lambda$ increases, more transport occurs and we asymptotically approach the balanced problem.

\begin{figure}
	\centering
	\includegraphics[width=0.99\linewidth]{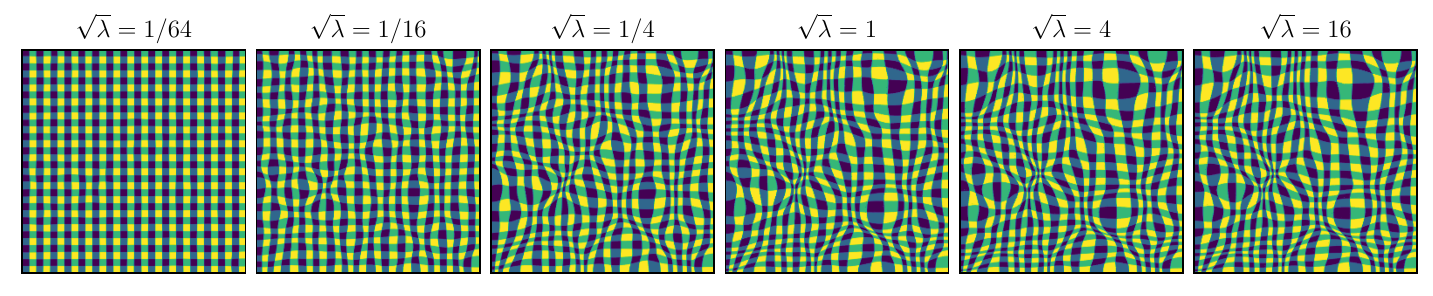}
	\caption{Influence of $\lambda$ in the transport plan. We color the density of each basic cell $Y$-marginal with respect to the total marginal, i.e.  $\nu_i / \proj_Y \pi$, with a color depending on its position in its respective $A$-composite cell. 
		The resulting grid pattern represents the target locations of mass in the basic cells $X_i$ according to the optimal transport plan. For small $\lambda$, mass is barely moving (the target grid is very close to the regular initial grid). As $\lambda$ increases, the grid is deformed further and further.
	}
	\label{fig:deformation-map}
\end{figure}

\paragraph{Runtime}
Figure \ref{fig:diff-time-lambda} shows the runtime of the numerical experiments, while Table \ref{tab:benchmark} reports subroutine runtime and solution quality for $\lambda = 1$.
We summarize the performance over all test problems by the median, since it is more robust to a few outliers where we observed convergence issues (see paragraph further below).

In Figure \ref{fig:diff-time-lambda}, left we can observe a steady increase in runtime for the \SinkhornGPU\ solver, followed by a sharp decline when certain value of $\lambda$ is crossed. 
This is consistent with the behavior shown in Figure \ref{fig:convergence-plateau}; the slightly different dependence in $\lambda$ is due to the fact that we are now using a multiscale solver with $\veps$-scaling, instead of the fixed scale and $\veps$ used in Figure \ref{fig:convergence-plateau}.
The runtime for \DomDecGPU\ is more consistent, with small values of $\lambda$ appearing to be the most challenging regime.
This can be explained by the fact that for small $\lambda$ we expect more non-local changes in mass, for which domain decomposition is not as well-suited as a global Sinkhorn solver.
Intermediate values of $\lambda$ show the largest benefit of using \DomDecGPU\ instead of \SinkhornGPU.

For very large $\lambda$ we are essentially solving a balanced problem (as exemplified by Figure \ref{fig:deformation-map}).
In this regime \DomDecGPU\ outperforms \SinkhornGPU\ on large problems, although it does not match the speed-up reported for balanced transport in \cite{hybrid}. This is mainly due to two reasons:
First, the unbalanced algorithm requires several steps that cause additional overhead, such as the Newton iteration in the Sinkhorn subsolver (see \eqref{eq:sinkhorn-Y-iteration} and discussion below), and the calculation of the global $Y$-marginal and $\nu_{-J}$ (we will comment on this further below). 
Besides, the balanced problem allows for a series of performance tricks (e.g.~batches of arbitrary size, clustering of the composite cells' sizes to maximize the efficiency of the Sinkhorn iteration) that are hard to adapt to unbalanced \DomDecGPU\ when employing the \texttt{staggered} parallelization strategy.

Figure \ref{fig:diff-time-lambda}, center shows the runtime in terms of $N$ for fixed $\lambda$. 
The scaling in $N$ of \DomDecGPU\ and \SinkhornGPU\ agrees well with that reported in \cite{hybrid} for the balanced case, with \DomDecGPU\ featuring a larger constant overhead for the reasons commented above, but a slower increase with $N$. Hence for large $N$ \DomDecGPU\ is faster than \SinkhornGPU.

\begin{figure}[h]
	\includegraphics[width=0.99\linewidth]{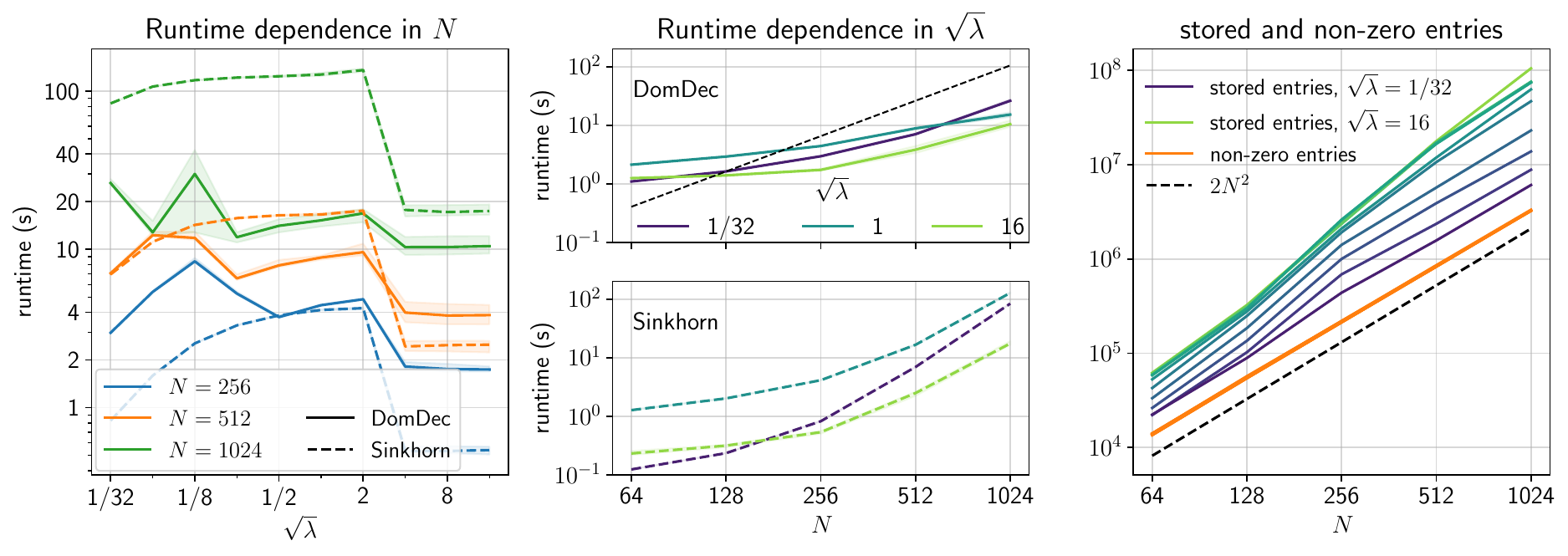}
	\caption{Left, Runtime comparison between \DomDecGPU\ and \SinkhornGPU, for fixed values of $N$. The solid line and shading represent the median and the 0.25 and 0.75 quantiles. Center, an analog comparison focusing on the dependence on $\lambda$. The dashed line is proportional to the marginal size, i.e.~$N^2$. Right, Number of non-zero entries and stored entries for different values of $N$ and $\lambda$. The interquantile range is not shown for the sake of clarity, but we observe it increases in magnitude with $\lambda$ until reaching the range portrayed in \cite{hybrid} for the balanced case.}
	\label{fig:diff-time-lambda}
\end{figure}

In Table \ref{tab:benchmark} we report more detailed statistics about the runtime and solution quality for $\lambda = 1$.
We observe that the Sinkhorn warm start ---used to solve the first multiscale layers--- claims a sizable share of the total runtime. Of the remaining time (when domain decomposition is running) about 60\% corresponds to the subproblem solving, 10\% to the computation of $\nu_{-J}$, about 8\% for manipulation of the sparse structure containing the cell $Y$-marginals $(\nu_i)_{i\in I}$, and another 8\% to the balancing procedure. Refinement, batching and other minor operations only account for a negligible fraction.
Within the solution of the cell subproblems, the Newton subroutine required to compute the $Y$-iteration (cf.~\eqref{eq:sinkhorn-Y-iteration} and discussion below) claims 47\% of the subproblem solving time, even after implementing it directly in CUDA.
It would be interesting to investigate if this routine can be improved further.

\begin{table}
	\caption{Summary of the median performance of \DomDecGPU\ and \SinkhornGPU\ and subroutines for $\lambda = 1$, over a total of 45 problem instances per $N$ (see Section \ref{sec:numerics-setup}).}
	\centering
	\small
\begin{tabular}{lrrrrr}
	\toprule
	image size  &    64$\times$64   &    128$\times$128  &    256$\times$256  &    512$\times$512  &     1024$\times$1024 \\
	\midrule
	\multicolumn{6}{l}{total runtime (seconds)}
	\\
	\midrule
	\DomDecGPU &     2.14 &     2.94 &     4.44 &    8.88 &     15.3 \\
	\SinkhornGPU&     1.28 &     2.03 &     4.14 &    16.6 &      127 \\ 
	\midrule
	\multicolumn{6}{l}{time spent in \DomDecGPU's subroutines (seconds)}
	\\
	\midrule
	warm up \SinkhornGPU & 0.915 &     1.78 &     3.06 &    6.14 &     6.21 \\
	Sinkhorn subsolver &  0.675 &    0.552 &    0.642 &    1.45 &     5.45 \\
	bounding box processing &    0.225 &    0.235 &     0.250 &   0.307 &     0.770 \\
	compute $\nu_{-J}$ &     0.120 &    0.121 &    0.128 &   0.183 &    0.883 \\
	\midrule
	\multicolumn{6}{l}{\DomDecGPU\ solution quality}
	\\
	\midrule
	$\LL^1$ marginal error $Y$ &  2.1e-04 &  8.6e-05 &  1.9e-05 & 4.2e-06 &  4.7e-06 \\
	$\LL^1$ marginal error $X$  &  5.5e-03 &  3.9e-03 &  3.4e-03 & 3.2e-03 &  3.2e-03 \\
	relative primal-dual gap &  1.0e-03 &  2.1e-03 &  3.7e-03 & 2.4e-03 &  3.5e-03 \\
	\midrule
	\multicolumn{6}{l}{\SinkhornGPU\ solution quality}
	\\
	\midrule
	$\LL^1$ marginal error $X$  &  6.3e-03 &  6.3e-03 &  6.3e-03 & 6.3e-03 &  6.3e-03 \\
	relative primal-dual gap &  1.1e-03 &  2.4e-03 &  4.5e-03 & 2.7e-03 &  3.7e-03 \\
	\bottomrule
\end{tabular}

	\label{tab:benchmark}
\end{table}

\paragraph{Issues with non-convergence}
In about 1\% of cell problems the subsolver in \DomDecGPU\ failed to reach the prescribed error goal and terminated before convergence.
This happened in cases where $\lambda$ was large and it is connected to the issue described in Figure \ref{fig:convergence-plateau}.
It was possible to restore convergence by further decreasing the tolerance of the warm start Sinkhorn, which indicates that the problem is connected to mass discrepancies in the initial plan. 

The \texttt{safe} safeguard was only triggered in the  1\% of failing instances mentioned above. In the other cases the greedy stage of the \texttt{staggered} strategy worked without issues. 
If the threshold for the \texttt{safe} method is lowered (which increases the frequency of its calls) more problems outside this 1\% cohort invoke the \texttt{safe} method at some stage. 
This does not hinder the Sinkhorn subsolver, but causes that more domain decomposition iterations are needed for convergence, since the \texttt{safe} method imposes a rather small step size.
Therefore, we recommend to use the algorithm with the settings described above and to only increase the precision goal for the Sinkhorn warm start when a problem with convergence occurs.

\paragraph{Solution quality}

Since the unbalanced optimality conditions \eqref{eq:optimality-conditions-KL} tend to the balanced optimality conditions as $\lambda \to \infty$, we call the $\LL^1$ difference between $\proj_X \pi$ and $e^{-\alpha/\lambda}\mu$ the \emph{$X$-marginal error}, and define the $Y$-marginal error analogously. Table \ref{tab:benchmark} gives the $X$-marginal error for \DomDecGPU\ and \SinkhornGPU\ and the $Y$-marginal error for \DomDecGPU\ (since \SinkhornGPU\ ends in a $Y$-iteration, its $Y$-marginal error is zero). We observe a similar $X$-marginal error for \DomDecGPU\ and \SinkhornGPU, with \DomDecGPU\ being slightly better. The $Y$-marginal error for \DomDecGPU\ is typically several orders of magnitude smaller than the $X$-marginal error. 

Table \ref{tab:benchmark} also reports the relative primal-dual gap, resulting from dividing the primal-dual gap \eqref{eq:PD-gap} by the dual score. Again the solution quality is very similar for \DomDecGPU\ and \SinkhornGPU. We observe an analogous trend in marginal error and solution quality for the other values of $\lambda$.

\paragraph{Sparsity}
As shown in Figure \ref{fig:diff-time-lambda}, right, the number of non-zero entries grows approximately linearly with the image resolution $N^2$, with an essentially identical behavior for all values of $\lambda$. This is similar to the balanced case as reported in \cite{BoSch2020,hybrid}.

On the other hand, the required memory increases with $\lambda$ up to the level reported in \cite{hybrid} for balanced domain decomposition on GPUs. This can be understood with the help of Figure \ref{fig:deformation-map}: for small $\lambda$ there is little actual transport happening and the support of the basic cell $Y$-marginals $(\nu_i)_i$ is close to that of the $X$-marginals $(\mu_i)_i$. 
As a result, the bounding box structure stores the cell $Y$-marginals quite efficiently, since their support is close to being square-shaped.
As $\lambda$ increases, the shapes of the cell $Y$-marginals become more complex and diverse as in the balanced problem, and the bounding box storage becomes less efficient. 

This trend is currently the limiting factor for further increasing the problem size in \DomDecGPU. 
We expect that this could be improved by more sophisticated storage methods for the cell $Y$-marginals, such as alternating between a box-based representation for the subproblem solving (enhancing the efficiency of the Sinkhorn iterations) and a more classical sparse representation for storage between iterations. 

\paragraph{Conclusion} 
Previous work has demonstrated the efficiency of domain decomposition for balanced optimal transport. 
Generalizing domain decomposition to unbalanced transport involved addressing both theoretical and numerical challenges.
Our numerical experiments show that unbalanced domain decomposition compares favorably to a global Sinkhorn solver for a wide range of values of $\lambda$ in medium to large problems. 
An important task for future work is an improved memory efficiency for the bounding box data structure and a faster inner Newton step in the Sinkhorn iteration.

\bibliography{manuscript}{}

\end{document}